\documentclass[12pt]{article}

\usepackage{amsmath}
\usepackage{amssymb}
\usepackage{amsfonts}
\usepackage{amsthm}
\usepackage{color,url}
\usepackage{dsfont}
\usepackage{setspace}
\usepackage{graphicx}
\usepackage[english]{babel}\usepackage[latin1]{inputenc}\usepackage{times}\usepackage[T1]{fontenc}
\usepackage{epstopdf}
\usepackage{epsfig}
\usepackage{float}
\usepackage{stmaryrd}
\usepackage{pifont}

\usepackage[shortlabels]{enumitem}

\newtheorem{thm}{\bf{Theorem}}[section]
\newtheorem{lem}[thm]{\bf{Lemma}}

\newtheorem{df}[thm]{\bf{Definition}}
\newtheorem{cor}[thm]{\bf{Corollary}}
\newtheorem{rem}[thm]{\bf{Remark}}

\newtheorem{prop}[thm]{\bf{Proposition}}
\newtheorem{Fact}[thm]{\bf{Fact}}
\newtheorem{ex}[thm]{\bf{Example}}

\numberwithin{equation}{section}					
\numberwithin{thm}{section}						

\newcommand{\dom}{\operatorname{dom}}
\newcommand{\intt}{\operatorname{int}}
\newcommand{\cl}{\operatorname{cl}}

\newcommand{\bdry}{\operatorname{bdry}}

\newcommand{\epig}{\operatorname{epi}}
\newcommand{\epi}{\stackrel{\operatorname{e}}{\rightarrow}}
\newcommand{\pw}{\stackrel{\operatorname{p}}{\rightarrow}}
\newcommand{\att}{\stackrel{\operatorname{f}}{\rightarrow}}

\newcommand{\Id}{\operatorname{Id}}

\newcommand{\elim}{\operatornamewithlimits{e-lim}}

\newcommand{\conv}{\operatornamewithlimits{conv}}

\newcommand{\argmin}{\operatornamewithlimits{argmin}}

\newcommand{\Prox}{\operatorname{Prox}}

\newcommand{\pavn}{\operatorname{\mathcal{PA}}}

\newcommand{\R}{\operatorname{\mathbb{R}}}
\newcommand{\N}{\operatorname{\mathbb{N}}}

\newcommand{\kkk}{\ensuremath{{k \in \N}}}

\newcommand{\aN}{\operatorname{\mathcal{N}_{\infty}}}

\newcommand{\bl}{\operatorname{\bar{\lambda}}}

\newcommand{\bmu}{\operatorname{\bar{\mu}}}
\newcommand{\balpha}{\operatorname{\bar{\alpha}}}
\newcommand{\vphi}{\ensuremath{\varphi^{\alpha}_{\mu}}}
\newcommand{\vphiba}{\ensuremath{\varphi^{\balpha}_{\mu}}}
\newcommand{\vphik}{\ensuremath{\varphi^{\alpha_{k}}_{\mu}}}
\newcommand{\vphifk}{\ensuremath{\varphi_{f_{k},g_{k},\alpha_{k},\mu_{k}}}}
\newcommand{\vphif}{\ensuremath{\varphi_{f,g,\alpha,\mu}}}
\newcommand{\vphibmu}{\ensuremath{\varphi_{f,g,\alpha,\bmu}}}

\newcommand{\fot}{\ensuremath{f_{1}\Box f_{2}}}
\newcommand{\RX}{\ensuremath{\,\left]-\infty,+\infty\right]}}
\newcommand{\RE}{\ensuremath{\,\left[-\infty,+\infty\right]}}

\newcommand{\lmin}{\ensuremath{\, \lambda_{\min}}}

\newcommand{\scal}[2]{\langle{{#1},{#2}}\rangle}

\newcommand{\jj}{\ensuremath{\,\mathfrak{q}}}
\newcommand{\pluss}{{{\,\text{\ding{57}}\,}}} 

\newcommand{\timess}{{{\,\text{\ding{75}}}}} 
\newcommand{\const}{\text{constant}}

\title{A proximal average for prox-bounded functions}
\author{
  J. Chen\thanks{ School of Mathematics and Statistics, Southwest University, Chongqing 400715, P. R. China.\ Email: \texttt{J.W.Chen713@163.com.}}, 
  X. Wang\thanks{Mathematics, University of British Columbia, Kelowna, B.C. V1V 1V7, Canada; and School of Mathematics and Statistics, Southwest University, Chongqing 400715, P. R. China.
\ Email: \texttt{shawn.wang@ubc.ca.}},~and
  C. Planiden\thanks{School of Mathematics and Applied Statistics, University of Wollongong, Wollongong, NSW, 2500, Australia.\ Email: \texttt{chayne@uow.edu.au.}}
  }
\date{\today}

\begin{document}

\maketitle

\begin{abstract}
In this work, we construct a proximal average for two prox-bounded functions, which recovers the classical
proximal average for two convex functions. The new proximal average transforms continuously in epi-topology from one proximal hull to the other. When one of the functions is differentiable, the new proximal average is differentiable. We give characterizations for Lipschitz and single-valued proximal mappings and we show that the convex combination of convexified proximal mappings is always
a proximal mapping. Subdifferentiability and behaviors of infimal values and minimizers are also studied.
\end{abstract}



\section{Introduction}

The proximal average provides a novel technique for averaging convex functions, see \cite{convmono,proxbas}.
 The proximal average has been used widely in applications such as machine learning \cite{reidwill,Yu13a}, optimization \cite{resaverage,wolenski,boyd14,planwang2016,zaslav}, matrix analysis \cite{kimlaws, lim18} and
modern monotone operator theory
\cite{simons}. The proximal mapping of the proximal average is precisely
the average of proximal mappings of the convex
functions involved. Averages of proximal mappings are
important in convex and nonconvex optimization algorithms;
see, e.g., \cite{convmono, aveproj}.
A proximal average for
 possible nonconvex functions has long been sought.

In this work, we have proposed a proximal average for prox-bounded functions, which enjoy rich
theory in variational analysis and optimization.
Our proximal average
significantly extends the works of \cite{proxbas} from convex functions to
possibly nonconvex functions. The new average function provides an epicontinuous transformation
between proximal hulls of functions, and reverts to the convex proximal average
definition in the case of convex functions.
When studying the proximal average of possibly nonconvex functions, two fundamental issues arise. The first is when
the proximal mapping is convex-valued; the second is when the function can
be recovered from its proximal mapping. It turns out that
resolving both difficulties requires the `proximal'
condition in variational analysis.

\subsection{Outline}

The plan of the paper is as follows. In the following three subsections, we give basic concepts from variational analysis,
review related work in the literature and state the blanket assumptions of the paper.
In Section \ref{s:prel}, we prove some interesting and new
properties of proximal functions, proximal mappings and envelopes. Section \ref{s:conv}
gives
an explicit relationship between the convexified proximal mapping and the Clarke
subdifferential of the Moreau envelope. Section \ref{s:char} provides
characterizations of Lipschitz and single-valued proximal mappings.
In Section \ref{s:main}, we
define the proximal average for prox-bounded functions and give a systematic study
of its properties.  Relationships to arithmetic average and epi-average and
full epi-continuity of the proximal average are studied
in Section \ref{s:rela}. Section \ref{s:opti} is devoted to optimal
value and minimizers and convergence in minimization of the proximal average.
In Section \ref{s:subd}, we investigate the subdifferentiability and
differentiability of the proximal average.
As an example, the
proximal average for quadratic functions is given in Section \ref{s:quad}.
Finally, Section \ref{s:theg} illustrates the difficulty when the proximal mapping
is not convex-valued.\par
Two distinguished features
of our proximal average
deserve to be singled out: whenever one of the function is differentiable,
the new proximal average is differentiable and the convex combinations of
convexified proximal mappings is always a proximal mapping.
While epi-convergence
\cite{attouch1984, beertopologies} plays a dominant role in our analysis of
 convergence
in minimization, the class of proximal functions, which
is significantly broader than the class of convex functions, is indispensable
for studying the proximal average.
In carrying out the proofs later, we often cite results from
the standard reference Rockafellar--Wets \cite{rockwets}.

\subsection{Constructs from variational analysis}

In order to define the proximal average of possibly nonconvex functions, we utilize the Moreau envelope
and proximal hull. In what follows, $\R^n$ is the $n$-dimensional Euclidean space
with Euclidean norm $\|x\|=\sqrt{\scal{x}{x}}$ and inner product
$\scal{x}{y}=\sum_{i=1}^{n}x_{i}y_{i}$ for $x,y\in\R^n$.

\begin{df}
For a proper function $f:\R^n\rightarrow\RX$ and parameters $0<\mu<\lambda$, the
{Moreau envelope}
function $e_{\lambda}f$ and {proximal mapping} are defined, respectively, by
$$e_{\lambda}f(x)=\inf_{w}\left\{f(w)+\frac{1}{2\lambda}\|w-x\|^2\right\},
\quad \Prox_{\lambda}f(x)=\argmin_{w}\left\{f(w)+\frac{1}{2\lambda}\|w-x\|^2\right\};$$
the {proximal hull} function $h_{\lambda}f$ is defined by
$$h_{\lambda}f(x)=\inf_{w}\left\{e_{\lambda}f(w)-\frac{1}{2\lambda}\|x-w\|^2\right\};$$
the {Lasry--Lions envelope}
 $e_{\lambda,\mu}f$ is defined by
$$e_{\lambda,\mu}f(x)=\sup_{w}\left\{e_{\lambda}f(w)-\frac{1}{2\mu}\|x-w\|^2\right\}.$$
\end{df}

\begin{df} The function $f:\R^n\rightarrow \RX$ is prox-bounded if
there exist $\lambda>0$ and $x\in \R^n$ such that
$e_{\lambda}f(x)>-\infty.$
The supremum of the set of all such $\lambda$ is the threshold $\lambda_{f}$
of prox-boundedness for $f$.
\end{df}
Any function $f:\R^{n}\rightarrow\RX$ that is bounded
below by an affine function has threshold of prox-boundedness $\lambda_{f}=\infty$; cf.
\cite[Example 3.28]{rockwets}. A differentiable function $f$ with a Lipschitz continuous
gradient has $\lambda_{f}>0$.

Our notation is standard. For every nonempty set $S\subset\R^n$, $\conv S$, $\cl S$ and $\iota_{S}$  denote the \emph{convex hull}, \emph{closure} and
\emph{indicator function} of set $S$, respectively.
For a proper, lower semicontinuous (lsc) function $f:\R^n\rightarrow\RX$, $\conv f$ is its convex hull and $f^*$ is its \emph{Fenchel conjugate}.
We let $\inf f$ and $\argmin f$ denote the infimum and
the set of minimizers of $f$ on $\R^n$, respectively. We call
$f$ \emph{level-coercive} if
$$\liminf_{\|x\|\rightarrow\infty}\frac{f(x)}{\|x\|}>0,$$
and \emph{coercive} if
$$\liminf_{\|x\|\rightarrow\infty}\frac{f(x)}{\|x\|}=\infty.$$
We use $\partial f$, $\hat{\partial }f, \partial_{L}f, \partial_{C}f$ for the Fenchel subdifferential, Fr\'echet subdifferential, limiting subdifferential
and Clarke subdifferential of $f$, respectively.  More precisely,
at a point $x\in\dom f$,  the \emph{Fenchel subdifferential}
of $f$ at $x$ is the set
$$\partial f(x)=\{s\in\R^n:\ f(y) \geq f(x)+\scal{s}{y-x} \text{ for all $y\in\R^n$}\};$$
the \emph{Fr\'echet subdifferential} of $f$ at $x$ is the set
$$\hat{\partial} f(x)=\{s\in\R^n:\ f(y) \geq f(x)+\scal{s}{y-x}+o(\|y-x\|)\};$$
the \emph{limiting subdifferential} of $f$ at $x$ is
$$\partial_{L}f(x)=\{v\in\R^n:\ \exists \text{ sequences } x_{k} \att x \text{ and }
s_{k}\in\hat{\partial}f(x_{k}) \text{ with } s_{k}\rightarrow v\},$$
where $x_{k} \att x$ means $x_{k}\rightarrow x$ and $f(x_{k})\rightarrow f(x)$.
We let $\Id:\R^n\rightarrow\R: x\mapsto x$ be the identity mapping and $\jj=\frac{1}{2}\|\cdot\|^2$.
The mapping $J_{\mu\partial_{L}f}=(\Id+\mu\partial_{L}f)^{-1}$ is
called the \emph{resolvent} of $\mu\partial_{L}f$;
cf. \cite[page 539]{rockwets}.
When $f$ is locally Lipschitz at $x$, the \emph{Clarke subdifferential}
 $\partial_{C}f$ at $x$ is $\partial_{C} f(x)=
\conv \partial_{L}f(x)$.
For further details on subdifferentials, see \cite{optanal,mordukhovich2006variational,rockwets}.
For $f_{1},f_{2}:\R^n\rightarrow\RX$, the \emph{infimal convolution}
 (or epi-sum)
of $f_{1}, f_{2}$ is defined by
$$(\forall x\in\R^n)\ f_{1}\Box f_{2}(x)=\inf_{w}\{f_{1}(x-w)+f_{2}(w)\},$$
and it is exact at $x$ if $\exists~w\in\R^n$ such that $f_{1}\Box f_{2}(x)=f_{1}(x-w)+f_{2}(w)$;
$f_{1}\Box f_{2}$ is exact if it is exact at every point of its domain.

\subsection{Related work}
A comparison to known work in the literature is in order.
In \cite{zhang2,zhang1}, Zhang et. al. defined a lower compensated convex
transform for $0<\mu<+\infty$ by
$$C_{\mu}^{l}(f)=\conv(2\mu\jj+f)-2\mu\jj.$$
The lower compensated convex transform is the proximal
hull.
In \cite{zhang2}, Zhang, Crooks and Orlando gave a comprehensive
study on the average compensated
convex approximation,
which is an arithmetic average of the proximal hull and the upper proximal hull.
While the proximal hull is a common ingredient,  our work and theirs
are completely different.
By nature, the proximal mapping of the
proximal average for convex functions
 is exactly the convex combination
of proximal mappings of individual convex functions \cite{proxbas}.
In \cite{proxave}, Hare proposed a proximal average by
$$\pavn_{1/\mu}=-e_{1/(\mu+\alpha(1-\alpha))}(-\alpha e_{1/\mu}f-(1-\alpha)e_{1/\mu}g).$$
For this average,
 $x\mapsto\pavn_{1/\mu}(x)$ is $\mathcal{C}^{1+}$ for every $\alpha\in]0,1[$, and enjoys other nice stabilities
 with respect to $\alpha$, see, e.g., \cite[Theorem 4.6]{proxave}. However,
this average definition has two disadvantages.\\\noindent (i) Even when both $f,g$ are convex, it does not
recover the proximal average for convex functions: $$-e_{1/\mu}(-\alpha e_{1/\mu}f-
(1-\alpha)e_{1/\mu}g).$$ (ii) Neither the proximal mapping
$\Prox_{1/(\mu+\alpha(1-\alpha))}\pavn_{1/\mu}$ nor
$\Prox_{1/\mu}\pavn_{1/\mu}$
is the average of the proximal mappings
$\Prox_{1/\mu}f$ and $\Prox_{1/\mu}g$.\par
In \cite{goebel2010proximal}, Goebel introduced a proximal average for saddle functions
by using extremal convolutions:
$$\mathcal{P}_{\mu,\eta}^{\cup\cap}
=\big(\lambda_{1}\timess (f_{1}+\mu\timess\jj_{x}-\eta\timess \jj_{y})\big)\pluss \big(\lambda_{2}\timess (f_{2}
+\mu\timess\jj_{x}-\eta\timess\jj_{y})\big)-
\mu\timess\jj_{x}+\eta\timess\jj_{y},$$
in which $f_{1}, f_{2}: \R^m\times \R^n\rightarrow\RE $ are
saddle functions, $\jj_{x}(x,y)=\jj(x), \jj_{y}(x,y)
=\jj(y)$, $\mu, \eta>0$, $\lambda_{1}+\lambda_{2}=1$ with $\lambda_{i}>0$, and
$\pluss$ is the extremal convolution.
Some nice results about self-duality with respect to saddle function conjugacy and partial conjugacy are put forth and proved by Goebel \cite{goebel2010proximal}.
Goebel's average is the proximal average for convex functions when each $f_{i}$
is convex. However, the proximal mapping of $\Prox_{\lambda}\mathcal{P}_{\mu,\eta}^{\cup\cap}$ is not
the convex combination of $\Prox_{\lambda}f_1$ and $\Prox_{\lambda} f_2$.

\subsection{Blanket assumptions}\label{s:assump}
Throughout the paper,
the functions $f,g:\R^n\rightarrow\RX$ are proper, lsc
and prox-bounded with thresholds $\lambda_{f}, \lambda_{g}>0$ respectively,
$\bl=\min\{\lambda_{f},\lambda_{g}\}$, $\lambda>0$,
$\mu>0$ and $\alpha\in [0,1]$.

\section{Preliminaries}\label{s:prel}
In this section, we
collect several facts and present some auxiliary results
on proximal mappings of proximal functions, Moreau envelopes
and proximal hulls, which will be used in the sequel.

\subsection{Relationship among three regularizations: $e_{\lambda}f$, $h_{\lambda}f$, and $e_{\lambda,\mu}f$}
Some key properties about these regularizations come as follows.
\begin{Fact}\emph{(\cite[Example 11.26]{rockwets})}\label{l:dcform}
Let $0<\lambda<\lambda_{f}$.
\begin{enumerate}[label=\rm(\alph*)]
\item\label{i:mor:l}
The Moreau envelope
$$e_{\lambda}f=
-\left(f+\frac{1}{2\lambda}\|\cdot\|^2\right)^*\bigg(\frac{\cdot}{\lambda}\bigg)+
\frac{1}{2\lambda}\|\cdot\|^2$$
is locally Lipschitz.
\item The proximal hull satisfies
$$h_{\lambda}f+\frac{1}{2\lambda}\|\cdot\|^2=\bigg(f+\frac{1}{2\lambda}\|\cdot\|^2\bigg)^{**}.$$
\end{enumerate}
\end{Fact}

\begin{Fact}\emph{(\cite[Examples 1.44, 1.46, Exercise 1.29]{rockwets})}\label{f:m-p-l}
Let $0<\mu<\lambda<\lambda_{f}$. One has
\begin{enumerate}[label=\rm(\alph*)]
\item $h_{\lambda}f=-e_{\lambda}(-e_{\lambda}f)$,
\item \label{i:p:hull}
$e_{\lambda} f=e_{\lambda}(h_{\lambda}f)$,
\item $h_{\lambda}(h_{\lambda}f)=h_{\lambda}f$,
\item\label{i:d:env}
 $e_{\lambda,\mu}f=-e_{\mu}(-e_{\lambda}f)=h_{\mu}(e_{\lambda-\mu}f)=e_{\lambda-\mu}(h_{\lambda}f)$,
\item $e_{\lambda_{1}}(e_{\lambda_{2}}f)=e_{\lambda_{1}+\lambda_{2}}f$ for
$\lambda_{1}, \lambda_{2}>0$.
\end{enumerate}
\end{Fact}

For more details about these regularizations, we refer the reader to
\cite{attouch1990approximation,infconv,proxhilbert,diffprop} and \cite[Chapter 1]{rockwets}.

\subsection{Proximal functions}

The concept of $\lambda$-proximal functions will play an important role. This
subsection is dedicated to properties of $\lambda$-proximal functions.

\begin{df} We say that $f$ is \emph{$\lambda$-proximal}
 if $f+\frac{1}{2\lambda}\|\cdot\|^2$ is convex.
\end{df}

\begin{lem}\label{l:env:neg}
\begin{enumerate}[label=\rm(\alph*)]
\item \label{i:e:1} The negative Moreau envelope $-e_{\lambda}f$ is always $\lambda$-proximal.
\item\label{i:e:2}
If $e_{\lambda}f$ is $\mathcal{C}^{1}$, then $f+\frac{1}{2\lambda}\|\cdot\|^2$ is convex,
i.e., $f$ is $\lambda$-proximal.
\end{enumerate}
\end{lem}
\begin{proof} By Fact~\ref{l:dcform},
\begin{equation}\label{e:moreau}
(\forall x\in\R^n)\ \frac{1}{2\lambda}\|x\|^2-e_{\lambda}f(x)=
\bigg(f+\frac{1}{2\lambda}\|\cdot\|^2\bigg)^{*}\bigg(\frac{x}{\lambda}\bigg).
\end{equation}
\ref{i:e:1}: This is clear from \eqref{e:moreau}.

\noindent\ref{i:e:2}: By \eqref{e:moreau}, the assumption ensures that
$\big(f+\frac{1}{2\lambda}\|\cdot\|^2\big)^{*}\big(\frac{x}{\lambda}\big)$ is differentiable.
It follows from Soloviov's theorem \cite{soloviov}
 that $f+\frac{1}{2\lambda}\|\cdot\|^2$ is convex.
\end{proof}

While for convex functions, proximal mappings and resolvents are the same, they
differ for nonconvex functions in general.

\begin{Fact}\emph{(\cite[Example 10.2]{rockwets})}
For any proper, lsc function $f:\R^n\rightarrow\RX$ and
any $\mu>0$, one has
$$(\forall x\in\R^n)\ P_{\mu}f(x)\subseteq J_{\mu\partial_{L}f}(x).$$
When $f$ is convex, the inclusion holds as an equation.
\end{Fact}
\noindent However, proximal functions have surprising properties.
\begin{prop}\label{p:resolventf}
Let $0<\mu<\lambda_{f}$. Then
the following are equivalent:
\begin{enumerate}[label=\rm(\alph*)]
\item\label{i:resolvent1}
 $\Prox_{\mu}f=J_{\mu\partial_{L}f}$,
\item\label{i:resolvent2} $f$ is $\mu$-proximal,
\item\label{i:resolvent3} $\Prox_{\mu}f$ is maximally monotone,
\item\label{i:resolvent4} $\Prox_{\mu}f$ is convex-valued.
\end{enumerate}
\end{prop}
\begin{proof}
\ref{i:resolvent2}$\Rightarrow$\ref{i:resolvent1}: See \cite[Proposition 12.19]{rockwets}
\& \cite[Example 11.26]{rockwets}.

\noindent\ref{i:resolvent1}$\Rightarrow$\ref{i:resolvent2}: As $\Prox_{\mu}f$ is always monotone,
$(\Prox_{\mu}f)^{-1}=(\Id+\mu\partial_{L}f)$ is monotone and it suffices to apply
\cite[Proposition 12.19(c)$\Rightarrow$(b)]{rockwets}.

\noindent\ref{i:resolvent2}$\Leftrightarrow$\ref{i:resolvent3}:
See \cite[Proposition 12.19]{rockwets}.

\noindent\ref{i:resolvent3}$\Rightarrow$\ref{i:resolvent4}: This is clear.

\noindent\ref{i:resolvent4}$\Rightarrow$\ref{i:resolvent3}: By \cite[Example 1.25]{rockwets},
$\Prox_{\mu}f$ is nonempty, compact-valued and monotone with full domain. As $\Prox_{\mu}f$ is convex-valued, it suffices to apply \cite{lohne08}.
\end{proof}

%

\begin{lem}\label{l:prox:map}
Let $f$ be $\lambda$-proximal and $0<\mu<\lambda$. Then
\begin{enumerate}[label=\rm(\alph*)]
\item\label{i:p:convex} $\Prox_{\lambda}f$ is convex-valued,
\item\label{i:p:single}  $\Prox_{\mu}f$ is single-valued.
\end{enumerate}
Consequently, $\Prox_{\mu}f$ is maximally monotone if $0<\mu\leq\lambda$.
\end{lem}
\begin{proof}
\ref{i:p:convex}:
Observe that
$$e_{\lambda}f(x)=\inf_{y}\left\{f(y)+\frac{1}{2\lambda}\|y\|^2-\langle\frac{x}{\lambda}, y
\rangle\right\}+\frac{1}{2\lambda}\|x\|^2.$$
Since
$f+\frac{1}{2\lambda}\|\cdot\|^2-\langle{\frac{x}{\lambda}},\cdot\rangle$ is convex, $\Prox_{\lambda}f(x)$ is convex.

\noindent\ref{i:p:single}: This follow from the fact that
$f+\frac{1}{2\mu}\|\cdot\|^2-\langle\frac{x}{\mu},\cdot\rangle$ is strictly convex
and coercive.

When $0<\mu<\lambda$, $\Prox_{\mu}f$ is continuous and monotone, so maximally monotone by
\cite[Example 12.7]{rockwets}. For the maximal monotonicity of $\Prox_{\lambda}f$, apply
\ref{i:p:convex} and
\cite{lohne08} or Lemma~\ref{l:prox:grad}.
\end{proof}

The set of proximal functions is a convex cone. In particular, one has the following.

\begin{prop}\label{p:p:cone}
Let $f_1$ be $\lambda_1$-proximal and $f_2$ be $\lambda_2$-proximal. Then for any $\alpha,\beta>0$, the function $\alpha f_1+\beta f_2$ is $\frac{\lambda_1\lambda_2}{\beta\lambda_1+\alpha\lambda_2}$-proximal.
\end{prop}
\begin{proof}
Since $f_1+\frac{1}{2\lambda_1}\|\cdot\|^2$ and $f_2+\frac{1}{2\lambda_2}\|\cdot\|^2$ are convex, so are $\alpha\left(f_1+\frac{1}{2\lambda_1}\|\cdot\|^2\right)$, $\beta\left(f_2+\frac{1}{2\lambda_2}\|\cdot\|^2\right)$ and their sum:
$$\alpha f_1+\beta f_2+\left(\frac{\alpha}{2\lambda_1}+\frac{\beta}{2\lambda_2}\right)\|\cdot\|^2=\alpha f_1+\beta f_2+\frac{\beta\lambda_1+\alpha\lambda_2}{2\lambda_1\lambda_2}\|\cdot\|^2.$$
Therefore, $\alpha f_1+\beta f_2$ is $\frac{\lambda_1\lambda_2}{\beta\lambda_1+\alpha\lambda_2}$-proximal.
\end{proof}

\subsection{The proximal mapping of the proximal hull}
\begin{lem}\label{l:e:h}Let $0<\lambda<\lambda_{f}$.
One has
\begin{equation}\label{i:hull:prox}
\Prox_{\lambda}(h_{\lambda}f)=\conv\Prox_{\lambda}f.
\end{equation}
\end{lem}
\begin{proof}
Applying \cite[Example 10.32]{rockwets} to
$-e_{\lambda} f=-e_{\lambda}(h_{\lambda}f)$ yields
$$\conv\Prox_{\lambda}(h_{\lambda}f)=\conv\Prox_{\lambda}f.$$
Since $h_{\lambda}$ is $\lambda$-proximal, by Lemma~\ref{l:prox:map} we have
$\conv\Prox_{\lambda}(h_{\lambda}f)=\Prox_{\lambda}(h_{\lambda}f).$
Hence \eqref{i:hull:prox} follows.
\end{proof}

\begin{lem} Let $0<\lambda<\lambda_{f}$. The following are equivalent:
\begin{enumerate}[label=\rm(\alph*)]
\item\label{p:hull} $\Prox_{\lambda}(h_{\lambda}f)=\Prox_{\lambda}f$,
\item\label{p:function} $f$ is $\lambda$-proximal.
\end{enumerate}
\end{lem}

\begin{proof}
\ref{p:hull}$\Rightarrow$\ref{p:function}: Since $\Prox_{\lambda}(h_{\lambda}f)=
\conv \Prox_{\lambda}(h_{\lambda}f)$, $\Prox_{\lambda} f$ is upper
semicontinuous, convex and compact valued, and monotone with full domain, so
maximally monotone in view of \cite{lohne08} or Lemma~\ref{l:prox:grad}.
By \cite[Proposition 12.19]{rockwets},
$f+\frac{1}{2\lambda}\|\cdot\|^2$ is convex, equivalently,
$f$ is $\lambda$-proximal by \cite[Example 11.26]{rockwets}.

\noindent\ref{p:function}$\Rightarrow$\ref{p:hull}: As $f$ is $\lambda$-proximal,
$\Prox_{\lambda}f$ is convex-valued by Lemma~\ref{l:prox:map}. Then Lemma~\ref{l:e:h} gives
$\Prox_{\lambda}(h_{\lambda}f)=\conv\Prox_{\lambda}f=\Prox_{\lambda}f.$
\end{proof}
\begin{cor} If $f\neq h_{\lambda}f$, then $\Prox_{\lambda}(h_{\lambda}f)
\neq\Prox_{\lambda}f.$
\end{cor}


\subsection{Proximal mappings and envelopes}
\begin{lem}\label{l:env:prox}
Let $0<\mu<\lambda<\bl$.
The following are equivalent:
\begin{enumerate}[label=\rm(\alph*)]
\item\label{i:env:fg}$e_{\lambda}f=e_{\lambda}g$,
\item\label{i:phull:fg}
 $h_{\lambda}f=h_{\lambda}g$,
\item \label{i:conv:fg} $\conv\left(f+\frac{1}{2\lambda}\|\cdot\|^2\right)=\conv\left(g+\frac{1}{2\lambda}\|\cdot\|^2\right)$,
\item\label{i:double:fg}
 $e_{\lambda,\mu}f=e_{\lambda,\mu}g$,

 \item\label{i:prox:initial}
  $\conv\Prox_{\lambda}f=\conv\Prox_{\lambda}g$, and for some $x_{0}\in\R^n$ one has
 $e_{\lambda}f(x_{0})=e_{\lambda}g(x_{0})$.

\end{enumerate}
Under any one of the conditions \ref{i:env:fg}--\ref{i:prox:initial}, one has
\begin{equation}\label{e:convexhull}
\overline{\conv} f=\overline{\conv} g.
\end{equation}
\end{lem}
\begin{proof}
\ref{i:env:fg}$\Rightarrow$\ref{i:phull:fg}:
We have
$-e_{\lambda}f=-e_{\lambda}g$ implies $-e_{\lambda}(-e_{\lambda}f)=-e_{\lambda}(-e_{\lambda}g)$,
which is \ref{i:phull:fg}.

\noindent\ref{i:phull:fg}$\Rightarrow$\ref{i:env:fg}:
This follows from $e_{\lambda}f=e_{\lambda}(h_{\lambda}f)=e_{\lambda}(h_{\lambda}g)=e_{\lambda}g.$

\noindent\ref{i:phull:fg}$\Leftrightarrow$\ref{i:conv:fg}: Since $\lambda<\bl$, we have that
$f+\frac{1}{2\lambda}\|\cdot\|^2$ and $g+\frac{1}{2\lambda}\|\cdot\|^2$ are coercive, so
$\conv\left(f+\frac{1}{2\lambda}\|\cdot\|^2\right)$ and $\conv\left(f+\frac{1}{2\lambda}\|\cdot\|^2\right)$
are lsc. Fact~\ref{l:dcform} gives
 $$h_{\lambda}f=\conv\bigg(f+\frac{1}{2\lambda}\|\cdot\|^2\bigg)-\frac{1}{2\lambda}\|\cdot\|^2,$$
 $$h_{\lambda}g=\conv\bigg(g+\frac{1}{2\lambda}\|\cdot\|^2\bigg)-\frac{1}{2\lambda}\|\cdot\|^2.$$

\noindent\ref{i:double:fg}$\Leftrightarrow$\ref{i:env:fg}: Invoking Fact~\ref{f:m-p-l}, we have
\begin{align*}
e_{\lambda,\mu}f =e_{\lambda,\mu}g &\Leftrightarrow h_{\mu}(e_{\lambda-\mu}f)=h_{\mu}(e_{\lambda-\mu}g)\\
& \Leftrightarrow e_{\mu}(h_{\mu}(e_{\lambda-\mu}f))=e_{\mu}(h_{\mu}(e_{\lambda-\mu}g))\\
& \Leftrightarrow e_{\mu}(e_{\lambda-\mu}f)=e_{\mu}(e_{\lambda-\mu}g)\\
& \Leftrightarrow e_{\lambda}f=e_{\lambda}g. \
\end{align*}

\noindent\ref{i:env:fg}$\Rightarrow$\ref{i:prox:initial}: The Moreau envelope $e_{\lambda}f(x)=e_{\lambda}g(x)$
for every $x\in\R^n$. Apply \cite[Example 10.32]{rockwets}
to $-e_{\lambda}f = -e_{\lambda}g$ to get
$$(\forall x\in\R^n)\ \frac{\conv \Prox_{\lambda}f(x)-x}{\lambda}=\frac{\conv \Prox_{\lambda}g(x)-x}{\lambda},
$$
which gives \ref{i:prox:initial} after simplifications.

\noindent\ref{i:prox:initial}$\Rightarrow$\ref{i:env:fg}: Since
both $e_{\lambda}f$ and $e_{\lambda}g$ are locally Lipschitz,
$\conv \Prox_{\lambda}f=\conv\Prox_{\lambda}g$
implies
$-e_{\lambda}f=-e_{\lambda}g+\const$ by
\cite[Example 10.32]{rockwets}. The $\const$ has to be zero by
$e_{\lambda}f(x_{0})=e_{\lambda}g(x_{0})$. Thus, \ref{i:env:fg} holds.

Equation~\eqref{e:convexhull} follows from the equivalence of \ref{i:env:fg}--\ref{i:double:fg}
and taking the Fenchel conjugate to $e_{\lambda}f=e_{\lambda}g$, followed
by cancelation of terms and taking the Fenchel conjugate again.
\end{proof}

The notion of `proximal' is instrumental.
\begin{cor}\label{c:needed1} Let $0<\mu\leq \lambda <\bl$, and
 let $f,g$ be $\lambda$-proximal.
Then $e_{\mu}f=e_{\mu}g$ if and only if $f=g$
\end{cor}
\begin{proof} Since $\mu\leq\lambda$, both $f,g$ are also $\mu$-proximal, so
$f=h_{\mu}f, g=h_{\mu}g$. Lemma~\ref{l:env:prox}\ref{i:env:fg}$
\Leftrightarrow$\ref{i:phull:fg} applies.
\end{proof}


\begin{prop}\label{p:needed1}
Let $0<\mu <\bl$, and let
$\Prox_{\mu}f=\Prox_{\mu}g$. If $f, g$ are $\mu$-proximal, then $f-g\equiv
\text{constant}$.
\end{prop}
\begin{proof}
As $\Prox_{\mu}f=\Prox_{\mu}g$,
by \cite[Example 10.32]{rockwets},
$\partial(-e_{\mu}f)=\partial(-e_{\mu}g)$. Since
both $-e_\mu f, -e_{\mu}g$ are locally Lipschitz and Clarke regular,
we obtain that there exists $-c\in\R$ such that
$-e_\mu f=-e_{\mu}g-c$. Because $f, g$ are $\mu$-proximal, we have
$$f=-e_{\mu}(-e_{\mu}f)=-e_{\mu}(-e_{\mu}g-c)=-e_{\mu}(-e_{\mu}g)+c=g+c,$$
as required.
\end{proof}

\subsection{An example}

The following example shows that one cannot remove the assumption of
$f, g$ being $\mu$-proximal in Proposition~\ref{p:resolventf}, Corollary~\ref{c:needed1}
and Proposition~\ref{p:needed1}.

\begin{ex}\label{e:proximal:fk}
 Consider the function
$$f_{k}(x)=\max\{0,(1+\varepsilon_{k})(1-x^2)\},$$
where $\varepsilon_{k}>0$.
It is easy to check that $f_{k}$ is $1/(2(1+\varepsilon_{k}))$-proximal, but not $1/2$-proximal.
\end{ex}


\noindent {\sl Claim 1: The functions $f_{k}$ have the same proximal mappings
and Moreau envelopes for all $k\in\N$. However, whenever $\varepsilon_{k_{1}}\neq\varepsilon_{k_{2}}$, $f_{k_{1}}-f_{k_{2}}=(\varepsilon_{k_{1}}-\varepsilon_{k_{2}})f\neq \const$.}

Indeed, simple calculus gives that for every $\varepsilon_{k}>0$ one has
\begin{equation*}\label{e:proximal2}
\Prox_{1/2}f_{k}(x)=\begin{cases}
x &\text{ if $x\geq 1$,}\\
1 &\text{ if $0<x<1$,}\\
\{-1,1\} &\text{ if $x=0$,}\\
-1 &\text{ if $-1<x<0$,}\\
x &\text{ if $x\leq -1$,}
\end{cases}
\end{equation*}
and
$$e_{1/2}f_{k}(x)=\begin{cases}
0 &\text{ if $x\geq 1$,}\\
(x-1)^2 &\text{ if $0\leq x<1$,}\\
(x+1)^2 &\text{ if $-1<x<0$,}\\
0 &\text{ if $x\leq -1$.}
\end{cases}
$$


\noindent{\sl Claim 2: $\Prox_{1/2}f_{k}\neq J_{1/2\partial_{L}f_{k}},$ i.e., the
proximal mapping
differs from the resolvent.}


Since $J_{1/2\partial_{L}f_{k}}=(\Id+1/2\partial_{L}f_{k})^{-1}$ and
$$\partial_{L}f_{k}(x)
=\begin{cases}
0 & \text{ if $x<-1$,}\\
[0,2(1+\varepsilon_{k})] &\text{ if $x=-1$,}\\
-2(1+\varepsilon_{k})x &\text{ if $-1<x<1$,}\\
[-2(1+\varepsilon_{k}),0] &\text{ if $x=1$,}\\
0  & \text{ if $x>1$},
\end{cases}
$$
we obtain
$$J_{1/2\partial_{L}f_{k}}(x)=
\begin{cases}
x & \text{ if $x<-1$,}\\
-1 &\text{ if $-1\leq x\leq \varepsilon_{k}$,}\\
-\frac{x}{\varepsilon_{k}} &\text{ if $-\varepsilon_{k}<x<\varepsilon_{k}$,}\\
1 &\text{ if $-\varepsilon_{k}\leq x\leq 1$,}\\
x & \text{ if $x>1$},
\end{cases}
$$
equivalently,
$$J_{1/2\partial_{L}f_{k}}(x)=
\begin{cases}
x & \text{ if $x<-1$,}\\
-1 &\text{ if $-1\leq x< -\varepsilon_{k}$,}\\
\left\{-1,-\frac{x}{\varepsilon_{k}},1\right\} &\text{ if $-\varepsilon_{k}\leq x\leq \varepsilon_{k}$,}\\
1 &\text{ if $\varepsilon_{k}< x\leq 1$,}\\
x & \text{ if $x>1$},
\end{cases}
$$
which does not equal \eqref{e:proximal2}.

\section{The convexified proximal mapping and Clarke subdifferential of the Moreau envelope}\label{s:conv}

The following result gives the relationship between the Clarke subdifferential
of the Moreau envelope and the convexified proximal mapping.
\begin{lem}\label{l:prox:grad}
 For $0<\mu<\lambda_{f}$, the following hold.
\begin{enumerate}[label=\rm(\alph*)]
\item \label{i:fplus:p}The convex hull
\begin{equation*}\label{e:convp}
\conv\Prox_{\mu}f=\partial \bigg(\mu f+\frac{1}{2}\|\cdot\|^2\bigg)^*.
\end{equation*}
In particular, $\conv\Prox_{\mu}f$ is maximally monotone.
\item\label{i:fplus:n}The limiting subdifferential
\begin{equation*}
-\partial_{L} \bigg(-\bigg(\mu f+\frac{1}{2}\|\cdot\|^2\bigg)^*\bigg)
\subseteq \Prox_{\mu}f.
\end{equation*}
\item \label{e:env:clarke} The Clarke subdifferential
\begin{equation}\label{e:clarkesub}
\partial_{C}(e_{\mu}f)=-\partial_{L}(-e_{\mu}f)=\frac{\Id-\conv \Prox_{\mu}f}{\mu}.
\end{equation}
If, in addition, $f$ is $\mu$-proximal, then
\begin{equation}\label{e:clarke:prox}
\partial_{C}(e_{\mu}f)=\frac{\Id-\Prox_{\mu}f}{\mu}.
\end{equation}
\end{enumerate}
\end{lem}

\begin{proof} \ref{i:fplus:p}: By Fact~\ref{l:dcform},
\begin{equation}\label{e:moreau:d}
-e_{\mu}f(x)=-\frac{1}{2\mu}\|x\|^2+\bigg(f+\frac{1}{2\mu}\|\cdot\|^2\bigg)^*\left(\frac{x}{\mu}\right).
\end{equation}
Using \cite[Example 10.32]{rockwets} and the subdifferential sum rule \cite[Corollary 10.9]{rockwets}, we get
$$
\frac{\conv\Prox_{\mu}f(x)-x}{\mu} =\partial_{L} (-e_{\mu}f)(x)=-\frac{x}{\mu}
+\partial \bigg(f+\frac{1}{2\mu}\|\cdot\|^2\bigg)^*\left(\frac{x}{\mu}\right).
$$
Simplification gives
\begin{align*}
\conv \Prox_{\mu}f(x) &
=\partial \mu \bigg(f+\frac{1}{2\mu}\|\cdot\|^2\bigg)^*\left(\frac{x}{\mu}\right)\\
& =\partial \bigg(\mu f+\frac{1}{2}\|\cdot\|^2\bigg)^*(x).
\end{align*}
Since $\mu f+\frac{1}{2}\|\cdot\|^2$ is coercive, we conclude that
$\left(\mu f+\frac{1}{2}\|\cdot\|^2\right)^*$ is a continuous convex function, so
$\conv \Prox_{\mu}f$ is maximally monotone \cite[Theorem 12.17]{rockwets}.

\noindent\ref{i:fplus:n}: By \eqref{e:moreau:d},
\begin{align*}
-\bigg(\mu f+\frac{1}{2}\|\cdot\|^2\bigg)^*(x) & =-\mu\bigg(f+\frac{1}{2\mu}\|\cdot\|^2\bigg)^*\left(\frac{x}{\mu}\right)\\
&=\mu e_{\mu}f(x)-\frac{1}{2}\|x\|^2.
\end{align*}
From \cite[Example 10.32]{rockwets} we obtain
\begin{align*}
\partial_{L} \bigg(-\bigg(\mu f+\frac{1}{2}\|\cdot\|^2\bigg)^*\bigg)
(x) &= \partial_{L} (\mu e_{\mu}f)(x)-x\\
&\subseteq \mu \frac{x-\Prox_{\mu}f(x)}{\mu}-x=-\Prox_{\mu}f(x).
\end{align*}
Therefore, $-\partial_{L} \bigg(-\bigg(\mu f+\frac{1}{2}\|\cdot\|^2\bigg)^*\bigg)(x)\subseteq \Prox_{\mu}f(x)$.

\noindent\ref{e:env:clarke}: As $-e_{\mu}f$ is Clarke regular, using \cite[Example 10.32]{rockwets}
we obtain
$$\partial_{C}e_{\mu}f(x)=-\partial_{C}(-e_{\mu}f)(x)=-\partial_{L} (-e_{\mu}f)(x)=
\frac{x-\conv \Prox_{\mu}f(x)}{\mu}.$$
If $f$ is $\mu$-proximal, then $\Prox_{\mu}f(x)$ is convex for every $x$, so
\eqref{e:clarke:prox} follows from \eqref{e:clarkesub}.
\end{proof}

\begin{rem} {\rm Lemma~\ref{l:prox:grad}\ref{i:fplus:p}} \& {\rm\ref{e:env:clarke}}
 extend {\rm\cite[Exercise 11.27]{rockwets}} and {\rm\cite[Theorem 2.26]{rockwets}},
 respectively, from convex functions to possibly nonconvex functions.
\end{rem}

It is tempting to ask whether
\begin{equation*}\label{e:boris}
\partial_L (e_{\mu}f)=\frac{\Id-\Prox_{\mu}f}{\mu}
\end{equation*}
holds. This is answered negatively
below.

\begin{prop} Let
$0<\lambda<\lambda_{f}$
and $\psi=h_{\lambda}f$. Suppose
that there exists $x_{0}\in\R^n$ such that $\Prox_{\lambda}f(x_{0})$ is not convex.
Then
\begin{equation}\label{e:sub:point}
\partial_{L} e_{\lambda}\psi(x_{0})\neq \frac{x_{0}-
\Prox_{\lambda}\psi(x_{0})}{\lambda};
\end{equation}
consequently,
$$\partial_{L} e_{\lambda}\psi\neq \frac{\Id-\Prox_{\lambda}\psi}{\lambda}.$$
\end{prop}

\begin{proof} We prove by contrapositive. Suppose \eqref{e:sub:point} fails,
i.e.,
\begin{equation}\label{e:prox:f}
\partial_{L} e_{\lambda}\psi(x_{0})=\frac{x_{0}-\Prox_{\lambda}\psi(x_{0})}{\lambda}.
\end{equation}
In view of $e_{\lambda}\psi=e_{\lambda}f$ and \cite[Example 10.32]{rockwets},
we have
\begin{equation}\label{e:prox:s}
\partial_{L} e_{\lambda}\psi(x_{0})= \partial_L e_{\lambda}f(x_{0})
\subseteq \frac{x_{0}-\Prox_{\lambda}f(x_{0})}{\lambda}.
\end{equation}
Since $\Prox_{\lambda}\psi=\conv\Prox_{\lambda}f$ by Lemma~\ref{l:e:h}, \eqref{e:prox:f} and \eqref{e:prox:s}
give
$$\frac{x_{0}-\conv\Prox_{\lambda}f(x_{0})}{\lambda}\subseteq \frac{x_{0}-\Prox_{\lambda}f(x_{0})}{\lambda},$$
which implies that $\Prox_{\lambda}f(x_{0})$ is a convex set.
This is a contradiction.
\end{proof}

\section{Characterizations of Lipschitz and single-valued proximal mappings}\label{s:char}
Simple examples show that proximal mappings can be wild, although always monotone.
\begin{ex} The function $f(x)=-\frac{1}{2}\|\cdot\|^2$ is prox-bounded with threshold
$\lambda_{f}=1$. We have $\Prox_{1}f=N_{\{0\}}$ the normal cone map at $0$,
i.e.,
$$N_{\{0\}}(x)=\begin{cases}
\R^n & \text{ if $x=0$,}\\
\varnothing & \text{ otherwise.}
\end{cases}
$$
When $0<\mu<1$, $$\Prox_{\mu}f=\frac{\Id}{1-\mu},$$
 which is Lipschitz continuous
with constant $1/(1-\mu)$.
\end{ex}

\begin{Fact}\emph{(\cite[Example 7.44]{rockwets})}
Let $f:\R^n\rightarrow \RX$ be proper, lsc and prox-bounded with threshold $\lambda_{f}$, and $0<\mu<\lambda_{f}$. Then
$\Prox_{\mu}f$ is always upper semicontinuous and locally bounded.
\end{Fact}

The following characterizations of the proximal mapping are of independent interest.

\begin{prop}[Lipschitz proximal mapping]
Let
$0<\mu<\lambda_{f}$.
Then the following are equivalent.
\begin{enumerate}[label=\rm(\alph*)]
\item\label{i:map} The proximal mapping $\Prox_{\mu}f$ is Lipschitz continuous with constant $\kappa>0$.
\item\label{i:function}
 The function
$$f+\frac{\kappa-1}{2\mu\kappa}\|\cdot\|^2$$
is convex.
\end{enumerate}
\end{prop}

\begin{proof}
\ref{i:map}$\Rightarrow$\ref{i:function}:  By Lemma~\ref{l:prox:grad}\ref{i:fplus:p},
$\bigg(\mu f+\frac{1}{2}\|\cdot\|^2\bigg)^*$ is differentiable and its
gradient is Lipschitz continuous with constant $\kappa$. By Soloviov's
theorem \cite{soloviov},
$\mu f+\frac{1}{2}\|\cdot\|^2$ is convex. Then the convex function
$\mu f+\frac{1}{2}\|\cdot\|^2$
has differentiable Fenchel conjugate $\big(\mu f+\frac{1}{2}\|\cdot\|^2\big)^*$ and
$\triangledown \big(\mu f+\frac{1}{2}\|\cdot\|^2\big)^*$ is Lipschitz continuous
with constant $\kappa$. It follows from \cite[Proposition 12.60]{rockwets} that
$\mu f+\frac{1}{2}\|\cdot\|^2$ is $\frac{1}{\kappa}$-strongly convex, i.e.,
$$\mu f+\frac{1}{2}\|\cdot\|^2-\frac{1}{\kappa}\frac{1}{2}\|\cdot\|^2$$
is convex. Equivalently,
$$f+\frac{\kappa-1}{2\mu\kappa}\|\cdot\|^2$$
is convex.

\noindent\ref{i:function}$\Rightarrow$\ref{i:map}: We have
$$\mu f+\frac{1}{2}\|\cdot\|^2-\frac{1}{\kappa}\frac{1}{2}\|\cdot\|^2$$
is convex, i.e.,
$\mu f+\frac{1}{2}\|\cdot\|^2$ is strongly convex with constant $\frac{1}{\kappa}$.
Then \cite[Proposition 12.60]{rockwets} implies that
$\left(\mu f+\frac{1}{2}\|\cdot\|^2\right)^*$ is differentiable and its
gradient is Lipschitz continuous with constant $\kappa$.
In view of Lemma~\ref{l:prox:grad}\ref{i:fplus:p},
$\Prox_{\mu}f$ is Lipschitz continuous with constant $\kappa$.
\end{proof}

\begin{cor}\label{c:lip}
 Let
$0<\mu<\lambda_{f}$.
Then the following are equivalent.
\begin{enumerate}[label=\rm(\alph*)]
\item\label{i:map1}The proximal mapping $\Prox_{\mu}f$ is Lipschitz continuous with constant $1$, i.e., nonexpansive.
\item\label{i:function1}
 The function
$f$
is convex.
\end{enumerate}
\end{cor}
\begin{df}\emph{(See \cite[Section 26]{rockconv} or \cite[page 483]{rockwets})}
 A proper, lsc, convex function $f:\R^n\rightarrow (-\infty, +\infty]$
is
\begin{enumerate}[label=\rm(\alph*)]
\item essentially strictly convex if $f$ is strictly convex on every convex subset
of $\dom \partial f$;
\item essentially differentiable if $\partial f(x)$ is a singleton whenever
$\partial f(x)\neq\varnothing$.
\end{enumerate}
\end{df}
\begin{prop}[single-valued proximal mapping]\label{p:single}
Let
$0<\mu<\lambda_{f}$.
Then the following are equivalent.
\begin{enumerate}[label=\rm(\alph*)]
\item\label{i:maps}The proximal mapping $\Prox_{\mu}f$ is single-valued, i.e., $\Prox_{\mu}f(x)$
is a singleton for every $x\in\R^n$.
\item\label{i:functionc}
 The function
$$f+\frac{1}{2\mu}\|\cdot\|^2$$
is essentially strictly convex and coercive.
\end{enumerate}
\end{prop}
\begin{proof}
\ref{i:maps}$\Rightarrow$\ref{i:functionc}:  By Lemma~\ref{l:prox:grad}\ref{i:fplus:p},
$\left(\mu f+\frac{1}{2}\|\cdot\|^2\right)^*$ is differentiable. By Soloviov's
theorem \cite{soloviov},
$\mu f+\frac{1}{2}\|\cdot\|^2$ is convex. The convex function
$\mu f+\frac{1}{2}\|\cdot\|^2$
has differentiable Fenchel conjugate $\left(\mu f+\frac{1}{2}\|\cdot\|^2\right)^*$. It follows from \cite[Proposition 11.13]{rockwets} that
$\mu f+\frac{1}{2}\|\cdot\|^2$ is essentially strictly convex.
Since $\left(\mu f+\frac{1}{2}\|\cdot\|^2\right)^*$ has full domain and
$\mu f+\frac{1}{2}\|\cdot\|^2$ is convex, the function
$\mu f+\frac{1}{2}\|\cdot\|^2$ is coercive by \cite[Theorem 11.8]{rockwets}.

\noindent\ref{i:functionc}$\Rightarrow$\ref{i:maps}:
Since
$\mu f+\frac{1}{2}\|\cdot\|^2$ is essentially strictly convex,
$\left(\mu f+\frac{1}{2}\|\cdot\|^2\right)^*$ is essentially differentiable by \cite[Theorem 11.13]{rockwets}.
Because $\mu f+\frac{1}{2}\|\cdot\|^2$ is coercive,
$\left(\mu f+\frac{1}{2}\|\cdot\|^2\right)^*$
has full domain. Then
$\left(\mu f+\frac{1}{2}\|\cdot\|^2\right)^*$ is differentiable on $\R^n$.
In view of Lemma~\ref{l:prox:grad}\ref{i:fplus:p},
$\Prox_{\mu}f(x)$ is single-valued for every $x\in\R^n$.
\end{proof}
Recall that for a nonempty, closed set $S\subseteq\R^n$ and every $x\in\R^n$,
the projection $P_{S}(x)$
consists of the points in $S$ nearest to $x$, so
$P_{S}=\Prox_{1}\iota_{S}$.
Combining Corollary~\ref{c:lip} and Proposition~\ref{p:single},
we can derive the following result
due to Rockafellar and Wets, \cite[Corollary 12.20]{rockwets}.
\begin{cor} Let $S$ be a nonempty, closed set in $\R^n$. Then the following are
equivalent:
\begin{enumerate}[label=\rm(\alph*)]
\item $P_{S}$ is single-valued,
\item $P_{S}$ is nonexpansive,
\item $S$ is convex.
\end{enumerate}
\end{cor}

\section{The proximal average for prox-bounded functions}\label{s:main}
The goal of this section is to establish a proximal average function that works for any two prox-bounded functions. Our framework
 will generalize the convex proximal average of
 \cite{proxpoint} to include nonconvex functions, in a manner
 that recovers the original definition in the convex case.

Remembering the standing assumptions in Subsection \ref{s:assump},
we define the \emph{proximal average} of $f, g$ associated with parameters $\mu, \alpha$ by
\begin{equation}\label{e:prox:def}
\vphi=-e_{\mu}(-\alpha e_{\mu}f-(1-\alpha)e_{\mu}g),
\end{equation}
which essentially relies on the Moreau envelopes.

\begin{thm}[basic properties of the proximal average]\label{t:prox}
Let
$0<\mu<\bl$,
and let $\vphi$ be defined as in \eqref{e:prox:def}.
Then the following hold.
\begin{enumerate}[label=\rm(\alph*)]
\item \label{i:env:conhull}
The Moreau envelope $e_{\mu}(\vphi)=\alpha e_{\mu}f+(1-\alpha)e_{\mu} g.$
\item\label{i:lowers}The proximal average $\vphi$ is proper, lsc and prox-bounded with threshold
$\lambda_{\vphi}\geq\bl$.
\item\label{i:epi:sum}
The proximal average $\vphi(x)=$
\begin{equation}\label{e:func}
\left[\alpha\conv\bigg(f+\frac{1}{2\mu}\|\cdot\|^2\bigg)\left(\frac{\cdot}{\alpha}\right)\Box (1-\alpha)\conv
\bigg(g+\frac{1}{2\mu}\|\cdot\|^2\bigg)\left(\frac{\cdot}{1-\alpha}\right)\right](x)
-\frac{1}{2\mu}\|x\|^2,
\end{equation}
where the inf-convolution $\Box$ is exact;
consequently, $\epig(\vphi+1/2\mu\|\cdot\|^2)=$
\begin{equation}\label{e:epig}
\alpha \epig\conv(f+1/2\mu\|\cdot\|^2)+(1-\alpha)\epig\conv(g+1/2\mu\|\cdot\|^2).
\end{equation}
\item\label{i:dom:convhull}
The domain $\dom\vphi=\alpha \conv\dom f+(1-\alpha)\conv\dom g$.
In particular, $\dom\vphi=\R^n$ if either one of $\conv\dom f$ and $\conv\dom g$ is
$\R^n$.
\item\label{i:hull:ave} The proximal average of $f$ and $g$ is the same
as the proximal average of proximal hulls $h_{\mu}f$ and $h_{\mu}g$, respectively.
\item\label{i:alpha}
When $\alpha=0$, $\varphi_{\mu}^{0}=h_{\mu}g$; when $\alpha=1$, $\varphi_{\mu}^{1}=h_{\mu}g$.
\item\label{i:phi:mu}
Each $\vphi$ is $\mu$-proximal, or equivalently, $\mu$-hypoconvex.
\item\label{i:f=g}
When $f=g$, $\vphi=h_{\mu}f$; consequently, $\vphi=f$ when $f=g$ is $\mu$-proximal.
\item\label{i:g=c}
When $g\equiv c\in\R$, $\vphi=e_{\mu/\alpha,\mu}(\alpha f+(1-\alpha)c)$,
the Lasry-Lions envelope of $\alpha f+(1-\alpha)c$.
\end{enumerate}
\end{thm}
\begin{proof}
\ref{i:env:conhull}:
Since $-\alpha e_{\mu}f-(1-\alpha)e_{\mu}g$ is $\mu$-proximal by
Lemma~\ref{l:env:neg}\ref{i:e:1} and Proposition~\ref{p:p:cone},
we have
\begin{align*}
-e_{\mu}(\vphi)& =-e_{\mu}(-e_{\mu}(-\alpha e_{\mu}f-(1-\alpha)e_{\mu}g))\\
&=h_{\mu}(-\alpha e_{\mu}f-(1-\alpha)e_{\mu}g)\\
&=-\alpha e_{\mu}f-(1-\alpha)e_{\mu}g.
\end{align*}
\noindent\ref{i:lowers}: Because
$0<\mu<\bl$,
both $e_{\mu}f$ and $e_{\mu}g$ are continuous, see, e.g., \cite[Theorem 1.25]{rockwets}.
By \ref{i:env:conhull}, $e_{\mu}(\vphi)$ is real-valued and continuous.  If
$\vphi$ is not proper, then $e_{\mu}(\vphi)\equiv-\infty$ or
$e_{\mu}(\vphi)\equiv\infty$, which is a contradiction. Hence,
$\vphi$ must be proper.
Lower semicontinuity follows from
the definition of the Moreau envelope.

To show that $\lambda_{\vphi}\geq \bl$, take any $\delta\in ]0,\bl-\mu[$. By \cite[Exercise 1.29(c)]{rockwets} and \ref{i:env:conhull}, we have
\begin{align*}
e_{\delta+\mu}(\vphi) &=e_{\delta}(e_{\mu}(\vphi))\\
& =e_{\delta}(\alpha e_{\mu}f+(1-\alpha)e_{\mu}g)\\
&\geq \alpha e_{\delta}(e_{\mu}f)+(1-\alpha)e_{\delta}(e_{\mu}g)\\
&=\alpha e_{\delta+\mu}f+(1-\alpha)e_{\delta+\mu}g>-\infty.
\end{align*}
Since $\delta\in ]0,\bl-\mu[$ was arbitrary, $\vphi$ has prox-bound
$\lambda_{\vphi}\geq \bl$.

\noindent\ref{i:epi:sum}: Since $\mu<\bl$, both
$e_{\mu}f$ and $e_{\mu}g$ are locally Lipschitz with full domain by
Fact~\ref{l:dcform}\ref{i:mor:l},
so
$$\dom \bigg(f+\frac{1}{2\mu}\|\cdot\|^2\bigg)^*=
\dom \bigg(g+\frac{1}{2\mu}\|\cdot\|^2\bigg)^*
=\R^n.$$ It follows from \cite[Theorem 11.23(a)]{rockwets} that
\begin{align*}
& \left[\alpha\bigg(f+\frac{1}{2\mu}\|\cdot\|^2\bigg)^*+
(1-\alpha)\bigg(g+\frac{1}{2\mu}\|\cdot\|^2\bigg)^*\right]^*\\
&=
\bigg(\alpha\bigg(f+\frac{1}{2\mu}\|\cdot\|^2\bigg)^*\bigg)^{*}\Box
\bigg((1-\alpha)\bigg(g+\frac{1}{2\mu}\|\cdot\|^2\bigg)^*\bigg)^{*}
\end{align*}
where the $\Box$ is exact; see, e.g., \cite[Theorem 16.4]{rockconv}.
By Fact~\ref{l:dcform},
\begin{align*}
& -\alpha e_{\mu}f-(1-\alpha)e_{\mu}g \\
&=\alpha\bigg(f+\frac{1}{2\mu}\|\cdot\|^2\bigg)^*\left(\frac{x}{\mu}\right)
+(1-\alpha)\bigg(g+\frac{1}{2\mu}\|\cdot\|^2\bigg)^*\left(\frac{x}{\mu}\right)-\frac{1}{2\mu}\|\cdot\|^2.
\end{align*}
Substitute this into the definition of $\vphi$ and use Fact~\ref{l:dcform} again
to obtain $\vphi(x)=$
\begin{align}
 &\left[\alpha\bigg(f+\frac{1}{2\mu}\|\cdot\|^2\bigg)^*\big(\frac{\cdot}{\mu}\big)+
(1-\alpha)\bigg(g+\frac{1}{2\mu}\|\cdot\|^2\bigg)^*\big(\frac{\cdot}{\mu}\big)
\right]^*\big(\frac{x}{\mu}\big)
-\frac{1}{2\mu}\|x\|^2\nonumber\\
=&\left[\alpha\bigg(f+\frac{1}{2\mu}\|\cdot\|^2\bigg)^*+
(1-\alpha)\bigg(g+\frac{1}{2\mu}\|\cdot\|^2\bigg)^*\right]^*\big(\mu\frac{x}{\mu}\big)
-\frac{1}{2\mu}\|x\|^2\nonumber\\
= & \left[\alpha\bigg(f+\frac{1}{2\mu}\|\cdot\|^2\bigg)^{**}\big(\frac{\cdot}{\alpha}\big)\Box
(1-\alpha)\bigg(g+\frac{1}{2\mu}\|\cdot\|^2\bigg)^{**}
\big(\frac{\cdot}{1-\alpha}\big)
\right](x)
-\frac{1}{2\mu}\|x\|^2\nonumber\\
=& \left[\alpha\conv\bigg(f+\frac{1}{2\mu}\|\cdot\|^2\bigg)\big(\frac{\cdot}{\alpha}\big)\Box
(1-\alpha)\conv\bigg(g+\frac{1}{2\mu}\|\cdot\|^2\bigg)(\frac{\cdot}{1-\alpha}\big)
\right](x)
-\frac{1}{2\mu}\|x\|^2,\label{e:box:exact}
\end{align}
in which
$$\bigg(f+\frac{1}{2\mu}\|\cdot\|^2\bigg)^{**}=\conv\bigg(f+\frac{1}{2\mu}\|\cdot\|^2\bigg)$$
$$\bigg(g+\frac{1}{2\mu}\|\cdot\|^2\bigg)^{**}=\conv\bigg(g+\frac{1}{2\mu}\|\cdot\|^2\bigg)$$
because $f+\frac{1}{2\mu}\|\cdot\|^2$ and $g+\frac{1}{2\mu}\|\cdot\|^2$
are coercive; see, e.g., \cite[Example 11.26(c)]{rockwets}.
Also, in \eqref{e:box:exact}, the infimal convolution is exact because
$\left(f+\frac{1}{2\mu}\|\cdot\|^2\right)^*$ and $\left(g+\frac{1}{2\mu}\|\cdot\|^2\right)^*$
have full domain and \cite[Theorem 16.4]{rockconv}
or \cite[Theorem 11.23(a)]{rockwets}.
\eqref{e:epig} follows from \eqref{e:func} and
\cite[Proposition 12.8(ii)]{convmono} or \cite[Exercise 1.28]{rockwets}.

\noindent\ref{i:dom:convhull}: This is immediate from \ref{i:epi:sum} and
\cite[Proposition 12.6(ii)]{convmono}.

\noindent\ref{i:hull:ave}: Use
\eqref{e:prox:def}, and the fact that
$e_{\mu}(h_{u}f)=e_{\mu}f$ and $e_{\mu}(h_{u}g)=e_{\mu}g$.

\noindent\ref{i:alpha}: When $\alpha=0$, this follows from $\varphi_{\mu}^{0}=-e_{\mu}(-e_{\mu}g)=h_{\mu}g$;
the proof for $\alpha=1$ case is similar.

\noindent\ref{i:phi:mu}: This follows from Fact~\ref{l:dcform}\ref{i:mor:l}.

\noindent\ref{i:f=g}: When $f=g$, we have $e_{\mu}\vphi=e_{\mu}f$ so that $-e_{\mu}\vphi=-e_{\mu}f$. Since
$\vphi$ is $\mu$-proximal by \ref{i:phi:mu}, it follows that
$\vphi=-e_{\mu}(-e_{\mu}\vphi)=-e_{\mu}(-e_{\mu}f)=h_{\mu}f$.

\noindent\ref{i:g=c}: This follows from
\begin{align*}
\vphi &=-e_{\mu}(-\alpha e_{\mu}f-(1-\alpha)c)=-e_{\mu}(-e_{\mu/\alpha}(\alpha f)-(1-\alpha)c)\\
&=-e_{\mu}[-e_{\mu/\alpha}(\alpha f+(1-\alpha)c)],
\end{align*}
and Fact~\ref{f:m-p-l}\ref{i:d:env}.
\end{proof}

\begin{prop}
\begin{enumerate}[label=\rm(\alph*)]
\item\label{i:regular} The proximal average $\vphi$ is always Clarke regular, prox-regular and strongly
amenable on $\R^n$.
\item\label{i:full:d}
 If one of the sets $\conv\dom f$ or $\conv\dom g$ is $\R^n$,
then $\vphi$ is locally Lipschitz on $\R^n$.
\item\label{i:u:proximable} When $f, g$ are both $\mu$-proximal, $\vphi$ is the proximal
average for convex functions.
\end{enumerate}
\end{prop}

\begin{proof}
One always has
$$\vphi=\bigg(\vphi+\frac{1}{2\mu}\|\cdot\|^2\bigg)-\frac{1}{2\mu}\|\cdot\|^2$$
where $\vphi+\frac{1}{2\mu}\|\cdot\|^2$ is convex
by Theorem~\ref{t:prox}\ref{i:phi:mu}.

\noindent\ref{i:regular}:
Use \cite[Example 11.30]{rockwets} and \cite[Exercise 13.35]{rockwets}
to conclude that $\vphi$ is prox-regular. \cite[Example 10.24(g)]{rockwets}
shows that $\vphi$ is strongly amenable.
Also, being a sum of a convex function and a $\mathcal{C}^2$ function, $\vphi$  is Clarke regular.

\noindent\ref{i:full:d}:
By Theorem~\ref{t:prox}\ref{i:dom:convhull},
$\dom\vphi=\R^n$, then $(\vphi+\frac{1}{2\mu}\|\cdot\|^2)$ is
a finite-valued convex function on $\R^n$, so it is
locally Lipschitz, hence $\vphi$.

\noindent\ref{i:u:proximable}: Since both $f+\frac{1}{2\mu}\|\cdot\|^2$ and
$g+\frac{1}{2\mu}\|\cdot\|^2$
are convex,  the result follows from
Theorem~\ref{t:prox}\ref{i:epi:sum}
and
\cite[Definition 4.1]{proxbas}.
\end{proof}

\begin{cor}
Let
$0<\mu<\bl$
and let $\vphi$ be defined as in \eqref{e:prox:def}.
Then
$$-\partial_L\left[-\bigg(\mu\vphi+\frac{1}{2}\|\cdot\|^2\bigg)^*\right]\subseteq
\alpha \Prox_{\mu}f+(1-\alpha)\Prox_{\mu}g.$$
\end{cor}
\begin{proof}
By Theorem~\ref{t:prox}\ref{i:env:conhull},
$e_{\mu}(\vphi)=\alpha e_{\mu}f+(1-\alpha)e_{\mu} g.$ Since both $e_{\mu}f, e_{\mu}g$ are locally
Lipschitz, the sum rule for $\partial_{L}$ \cite[Corollary 10.9]{rockwets} gives
\begin{align*}
\partial_{L} e_{\mu}\vphi(x) & \subseteq \alpha\partial_{L} e_{\mu}f(x)
+(1-\alpha)\partial_{L} e_{\mu}g(x)\\
&\subseteq \alpha \frac{x-\Prox_{\mu}f(x)}{\mu}+(1-\alpha)\frac{x-\Prox_{\mu}g(x)}{\mu}\\
&=\frac{x}{\mu}-\frac{\alpha \Prox_{\mu}f(x)+(1-\alpha)\Prox_{\mu}g(x)}{\mu},
\end{align*}
from which
$$\partial_{L}\left(e_{\mu}\vphi-\frac{1}{2\mu}\|x\|^2\right)\subseteq -\frac{\alpha \Prox_{\mu}f(x)+(1-\alpha)\Prox_{\mu}g(x)}{\mu}.$$
As
$$e_{\mu}\vphi(x)-\frac{1}{2}\|x\|^2=-\left(\vphi+\frac{1}{2\mu}\|\cdot\|^2\right)^{*}\left(\frac{x}{\mu}\right)
=-\frac{\left(\mu\vphi+\frac{1}{2}\|\cdot\|^2\right)^*(x)}{\mu},$$
we have
$$-\partial_L\left (-\left(\mu\vphi+\frac{1}{2}\|\cdot\|^2\right)^{*}\right)(x)\subseteq\alpha \Prox_{\mu}f(x)+(1-\alpha)\Prox_{\mu}g(x).$$
\end{proof}

A natural question to ask is whether $\alpha \Prox_{\mu}f+(1-\alpha)\Prox_{\mu}g$ is still a
proximal mapping. Although this is not clear in general, we have the following.

\begin{thm}[the proximal mapping of the proximal average]\label{prop:convcomb}
Let $0<\mu<\bl$ and let $\vphi$ be defined as in \eqref{e:prox:def}.
Then
\begin{equation}\label{e:prox:conv}
\Prox_{\mu}\varphi^\alpha_{\mu}
=\alpha\conv\Prox_{\mu}f+(1-\alpha)\conv\Prox_{\mu}g.
\end{equation}
\begin{enumerate}[label=\rm(\alph*)]
\item\label{i:u:prox} When both $f$ and $g$ are $\mu$-proximal, one has
$$\Prox_{\mu}\varphi^\alpha_{\mu}
=\alpha\Prox_{\mu}f+(1-\alpha)\Prox_{\mu}g.$$
\item\label{i:r:prox} Suppose that on an open subset $U\subset\R^n$
both $\Prox_{\mu}f, \Prox_{\mu}g$ are
single-valued (e.g., when
$e_{\mu}f$ and $e_{\mu}g$ are
continuously differentiable).
Then $\Prox_{\mu}\varphi^\alpha_{\mu}$ is single-valued, and
$$\Prox_{\mu}\varphi^\alpha_{\mu}
=\alpha\Prox_{\mu}f+(1-\alpha)\Prox_{\mu}g \text{ on $U$.}$$
\item\label{i:r:prox2} Suppose that on an open subset $U\subset\R^n$
both $\Prox_{\mu}f, \Prox_{\mu}g$ are
single-valued and Lipschitz continuous (e.g., when $f$ and $g$ are prox-regular).
Then $\Prox_{\mu}\varphi^\alpha_{\mu}$ is single-valued and Lipschitz continuous, and
$$\Prox_{\mu}\varphi^\alpha_{\mu}
=\alpha\Prox_{\mu}f+(1-\alpha)\Prox_{\mu}g \text{ on $U$.}$$
\end{enumerate}
\end{thm}

\begin{proof}
By Theorem~\ref{t:prox},
$$-e_{\mu}(\varphi^\alpha_{\mu})=-\alpha e_{\mu}f-(1-\alpha)e_{\mu}g.$$
Since both $-e_{\mu}f, -e_{\mu}g$ are Clarke regular, the sum rule \cite[Corollary 10.9]{rockwets}
gives
$$\partial_L(-e_{\mu}(\varphi^\alpha_{\mu}))=\alpha\partial_L (- e_{\mu}f)
+(1-\alpha)\partial_L(-e_{\mu}g).$$
Apply \cite[Example 10.32]{rockwets} to get
$$\frac{\conv\Prox_{\mu}\vphi(x)-x}{\mu}=\alpha \frac{\conv\Prox_{\mu}f(x)-x}{\mu}+
(1-\alpha)\frac{\conv\Prox_{\mu}g(x)-x}{\mu}$$
from which
$$
\conv\Prox_{\mu}\varphi^\alpha_{\mu}
=\alpha\conv\Prox_{\mu}f+(1-\alpha)\conv\Prox_{\mu}g.
$$
Since $\vphi$ is $\mu$-proximal, $\conv\Prox_{\mu}\vphi=\Prox_{\mu}\vphi$,
therefore, \eqref{e:prox:conv} follows.

\noindent\ref{i:u:prox}: Since $f,g$ are $\mu$-proximal,
$\Prox_{\mu}f$ and $\Prox_{\mu}g$ are convex-valued
by Proposition~\ref{p:resolventf}.

\noindent\ref{i:r:prox}: When $e_{\mu}f$ and $e_{\mu}g$ are continuously differentiable,
both $\Prox_{\mu}f, \Prox_{\mu}g$ are
single-valued on $U$ by \cite[Proposition 5.1]{proxhilbert}.

\noindent\ref{i:r:prox2}: When $f$ and $g$ are prox-regular on $U$,
both $\Prox_{\mu}f, \Prox_{\mu}g$ are
single-valued and Lipschitz continuous on $U$ by \cite[Proposition 5.3]{proxhilbert}
or \cite[Proposition 13.37]{rockwets}.
\end{proof}

\begin{cor}
Let $0<\mu<\bl$ and
let $\vphi$ be defined as in \eqref{e:prox:def}.
Then
$$
\Prox_{\mu}\varphi^\alpha_{\mu}
=\alpha\Prox_{\mu}(h_{\mu}f)+(1-\alpha)\Prox_{\mu}(h_{\mu}g).
$$
\end{cor}
\begin{proof}
Combine Theorem~\ref{prop:convcomb}
and Lemma~\ref{l:e:h}.
\end{proof}

\begin{cor} Let $\mu>0$. The following set of proximal mappings
$$\{\Prox_{\mu}f|\ \text{$f$ is $\mu$-proximal and $\mu<\lambda_{f}$}\}$$
is a convex set. Moreover, for every $\mu$-proximal function,
$\Prox_{\mu}f=(\Id+\mu\partial_{L} f)^{-1}$.
\end{cor}
\begin{proof}
Apply Theorem~\ref{prop:convcomb}\ref{i:u:prox},
Theorem~\ref{t:prox}\ref{i:lowers}\&\ref{i:phi:mu} and Proposition~\ref{p:resolventf}.
\end{proof}

\section{Relationships to the arithmetic average and epi-average}\label{s:rela}

\begin{df}[epi-convergence and epi-topology]
\emph{(See \cite[Chapter~6]{rockwets}.)}
Let $f$ and $(f_k)_\kkk$ be functions from $\R^n$ to $\RX$. Then
$(f_k)_{k\in\N}$ \emph{epi-converges} to $f$, in symbols $f_k\epi f$,
if for every $x\in \R^n$ the following hold: \hfill
\begin{enumerate}[label=\rm(\alph*)]
\item $\big(\forall\,(x_k)_{\kkk}\big)$ $x_k\to x \Rightarrow
f(x)\leq\liminf f_k(x_k)$;
\item $\big(\exists (y_k)_\kkk\big)$ $y_k\to x$
and $\limsup f_k(y_k) \leq f(x)$.
\end{enumerate}
We write $\elim_{k\rightarrow\infty}f_{k}=f$ to say that $f_k$ epi-converges to $f$.
The \emph{epi-topology} is the topology induced by epi-convergence.
\end{df}
\begin{rem} The threshold $\bl=+\infty$ whenever both $f,g$ are bounded
from below by an affine function.\end{rem}
\begin{thm}\label{t:go:infinity} Let $0<\mu<\bl$.
One has the following.
\begin{enumerate}[label=\rm(\alph*)]
\item\label{i:mono:phi} For every fixed $x\in \R^n$,
the function $\mu\mapsto\vphi(x)$ is monotonically decreasing and
left-continuous on $]0,\bl]$.

\item\label{i:inf:phi1}
The pointwise limit
$\lim_{\mu\uparrow \bl}\vphi=\inf_{\bl>\mu>0}\vphi=$
$$
\left[\alpha\conv\bigg(f+\frac{1}{2\bl}\|\cdot\|^2\bigg)\bigg(\frac{\cdot}{\alpha}\bigg)\Box (1-\alpha)\conv
\bigg(g+\frac{1}{2\bl}\|\cdot\|^2\bigg)\bigg(\frac{\cdot}{1-\alpha}\bigg)\right](x)
-\frac{1}{2\bl}\|x\|^2.
$$

\item\label{i:inf:phi2} When $\bl=\infty$,
the pointwise limit
\begin{equation}\label{e:pointwise}
\lim_{\mu\uparrow \infty}\vphi=\inf_{\mu>0}\vphi=
\alpha\conv f\bigg(\frac{\cdot}{\alpha}\bigg)\Box (1-\alpha)\conv
g\bigg(\frac{\cdot}{1-\alpha}\bigg), \text{ and }
\end{equation}
the epigraphical limit
\begin{equation}\label{i:epi:limit}
\elim_{\mu\uparrow\infty}\vphi=
\cl\left[\alpha\conv f\bigg(\frac{\cdot}{\alpha}\bigg)\Box (1-\alpha)\conv
g\bigg(\frac{\cdot}{1-\alpha}\bigg)\right].
\end{equation}
\end{enumerate}
\end{thm}
\begin{proof}
\ref{i:mono:phi}:
We have $\vphi(x)=$
\footnotesize\begin{align}
&
\inf_{u+v=x}\left(\alpha\inf_{\sum_{i}\alpha_{i}x_{i}=u/\alpha\atop{\sum_{i}\alpha_{i}=1,
\alpha_{i}\geq 0}}\left(\sum_{i}\alpha_{i}f(x_{i})
+\alpha_{i}\frac{1}{2\mu}\|x_{i}\|^2\right)+(1-\alpha)\inf_{\sum_{j}\beta_{j}y_{j}=v/(1-\alpha)
\atop{\sum_{j}\beta_{j}=1,
\beta_{j}\geq 0}}\left(\sum_{j}\beta_{j}g(y_{j})
+\beta_{j}\frac{1}{2\mu}\|y_{j}\|^2\right)\right)\nonumber\\
&\quad -\frac{1}{2\mu}\|x\|^2\nonumber\\
&=
\inf_{{\alpha\sum_i\alpha_{i}x_{i}+
(1-\alpha)\sum_{j}\beta_{j}y_{j}=x}\atop{\sum_{i}\alpha_{i}=1,
\sum_{j}\beta_{j}=1},\alpha_{i}\geq 0, \beta_{j}\geq 0}\bigg(\alpha\sum_{i}\alpha_{i}f(x_{i})
+(1-\alpha)\sum_{j}\beta_{j}g(y_{j})+\nonumber\\
&\quad \frac{1}{2\mu}\underbrace{\left(\alpha\sum_{i}\alpha_{i}\|x_{i}\|^2
+(1-\alpha)\sum_{j}\beta_{j}\|y_{j}\|^2-\bigg\|\alpha\sum_{i}\alpha_{i}x_{i}+(1-\alpha)\sum_{j}
\beta_{j}y_{j}\bigg\|^2\right)}\bigg).\nonumber
\end{align}\normalsize
The underbraced part is nonnegative because $\|\cdot\|^2$ is convex, $\sum_{i}\alpha_{i}=1,
\sum_{j}\beta_{j}=1$. It follows that
$\mu\mapsto\vphi$ is a monotonically decreasing function on $]0,+\infty[$.

Let $\bmu\in ]0,\bl]$. Then $\lim_{\mu\uparrow\bmu}\vphi
=\inf_{\bmu>\mu>0}\vphi=$
\begin{align}
&\inf_{\bmu>\mu>0}
\inf_{{\alpha\sum_i\alpha_{i}x_{i}+
(1-\alpha)\sum_{j}\beta_{j}y_{j}=x}\atop{\sum_{i}\alpha_{i}=1,
\sum_{j}\beta_{j}=1},\alpha_{i}\geq 0, \beta_{j}\geq 0}\bigg(\alpha\sum_{i}\alpha_{i}f(x_{i})
+(1-\alpha)\sum_{j}\beta_{j}g(y_{j})+\nonumber\\
&\quad \frac{1}{2\mu}\left(\alpha\sum_{i}\alpha_{i}\|x_{i}\|^2
+(1-\alpha)\sum_{j}\beta_{j}\|y_{j}\|^2-\bigg\|\alpha\sum_{i}\alpha_{i}x_{i}+(1-\alpha)\sum_{j}
\beta_{j}y_{j}\bigg\|^2\right)\bigg)\label{e:k1}\\
&=
\inf_{{\alpha\sum_i\alpha_{i}x_{i}+
(1-\alpha)\sum_{j}\beta_{j}y_{j}=x}\atop{\sum_{i}\alpha_{i}=1,
\sum_{j}\beta_{j}=1},\alpha_{i}\geq 0, \beta_{j}\geq 0}\inf_{\bmu>\mu>0}\bigg(\alpha\sum_{i}\alpha_{i}f(x_{i})
+(1-\alpha)\sum_{j}\beta_{j}g(y_{j})+\nonumber\\
&\quad \frac{1}{2\mu}\left(\alpha\sum_{i}\alpha_{i}\|x_{i}\|^2
+(1-\alpha)\sum_{j}\beta_{j}\|y_{j}\|^2-\bigg\|\alpha\sum_{i}\alpha_{i}x_{i}+(1-\alpha)\sum_{j}
\beta_{j}y_{j}\bigg\|^2\right)\bigg)\\
&=
\inf_{{\alpha\sum_i\alpha_{i}x_{i}+
(1-\alpha)\sum_{j}\beta_{j}y_{j}=x}\atop{\sum_{i}\alpha_{i}=1,
\sum_{j}\beta_{j}=1},\alpha_{i}\geq 0, \beta_{j}\geq 0}
\bigg(\alpha\sum_{i}\alpha_{i}f(x_{i})
+(1-\alpha)\sum_{j}\beta_{j}g(y_{j})+\nonumber\\
&\quad \frac{1}{2\bmu}\left(\alpha\sum_{i}\alpha_{i}\|x_{i}\|^2
+(1-\alpha)\sum_{j}\beta_{j}\|y_{j}\|^2-\bigg\|\alpha\sum_{i}\alpha_{i}x_{i}+(1-\alpha)\sum_{j}
\beta_{j}y_{j}\bigg\|^2\right)\bigg)\label{e:k2}\\
&=
\left[\alpha\conv\bigg(f+\frac{1}{2\bmu}
\|\cdot\|^2\bigg)\bigg(\frac{\cdot}{\alpha}\bigg)\Box (1-\alpha)\conv
\bigg(g+\frac{1}{2\bmu}\|\cdot\|^2\bigg)\bigg(\frac{\cdot}{1-\alpha}\bigg)\right](x)
-\frac{1}{2\bmu}\|x\|^2\nonumber.
\end{align}
\noindent\ref{i:inf:phi1}: This follows from \ref{i:mono:phi}.

\noindent\ref{i:inf:phi2}:
By \ref{i:mono:phi}, we have $\lim_{\mu\rightarrow\infty}\vphi
=\inf_{\mu>0}\vphi$. Using similar arguments as \eqref{e:k1}--\eqref{e:k2},
we obtain $\inf_{\mu>0}\vphi=$
\begin{align*}
&\inf_{{\alpha\sum_i\alpha_{i}x_{i}+
(1-\alpha)\sum_{j}\beta_{j}y_{j}=x}\atop{\sum_{i}\alpha_{i}=1,
\sum_{j}\beta_{j}=1},\alpha_{i}\geq 0, \beta_{j}\geq 0}\bigg(\alpha\sum_{i}\alpha_{i}f(x_{i})
+(1-\alpha)\sum_{j}\beta_{j}g(y_{j})\bigg)\\
&=\inf_{{u+v=x}}\bigg(\alpha\inf_{\sum_i\alpha_{i}x_{i}=u/\alpha
\atop{\sum_{i}\alpha_{i}=1,\alpha_{i}\geq 0}}\sum_{i}\alpha_{i}f(x_{i})
+(1-\alpha)\inf_{\sum_{j}\beta_{j}y_{j}=v/(1-\alpha)\atop{\sum_{j}\beta_{j}=1}, \beta_{j}\geq 0}\sum_{j}\beta_{j}g(y_{j})\bigg)\\
&=\inf_{{u+v=x}}\bigg(\alpha(\conv f)(u/\alpha)
+(1-\alpha)(\conv g)(v/(1-\alpha))\bigg),
\end{align*}
as required. To get \eqref{i:epi:limit}, we combine \eqref{e:pointwise} and
\cite[Proposition 7.4(c)]{rockwets}.
\end{proof}

In order to study the limit behavior when $\mu\downarrow 0$,  a lemma helps.
We omit its simple proof.

\begin{lem}\label{l:concave:e}
The Moreau envelope function respects the inequality$$e_{\mu}(\alpha f_{1}+(1-\alpha)f_{2})\geq \alpha e_{\mu}f_{1}+(1-\alpha)e_{\mu}f_{2}.$$
\end{lem}

\begin{thm}\label{t:u:0}
Let $0<\mu<\bl$. One has
\begin{enumerate}[label=\rm(\alph*)]
\item\label{i:three:c}
\begin{equation}\label{e:three:c}
\alpha e_{\mu}f+(1-\alpha) e_{\mu}g \leq\vphi\leq \alpha h_{\mu}f+(1-\alpha)h_{\mu}g
\leq\alpha f+(1-\alpha) g \text{ and }
\end{equation}
\item when $\mu\downarrow 0$, the pointwise limit and epi-graphical limit agree with
\begin{equation}\label{e:lim:0}
\lim_{\mu\downarrow 0}\vphi=\sup_{\mu>0}\vphi=\alpha f+(1-\alpha)g.
\end{equation}
Furthermore, the convergence in \eqref{e:lim:0}
 is uniform on compact subsets of $\R^n$ when $f,g$ are continuous.
\end{enumerate}
\end{thm}
\begin{proof}
Apply Lemma~\ref{l:concave:e} with $f_{1}=-e_{\mu}f, f_{2}=-e_{\mu}g$ to obtain
$e_{\mu}(\alpha(-e_{\mu}f)+(1-\alpha)(-e_{\mu}g))\geq \alpha e_{\mu}(-e_{\mu}f)+(1-\alpha)
e_{\mu}(-e_{\mu}g).$ Then
\begin{equation}\label{e:1}
\vphi\leq \alpha (-e_{\mu}(-e_{\mu}f))+(1-\alpha)(-e_{\mu}(-e_{\mu}g))=
\alpha h_{\mu}f+(1-\alpha)h_{\mu}g.
\end{equation}
On the other hand,
$e_{\mu}(\alpha(-e_{\mu}f)+(1-\alpha)(-e_{\mu}g))\leq \alpha(-e_{\mu}f)+(1-\alpha)(-e_{\mu}g)$
so
\begin{equation}\label{e:2}
\vphi\geq \alpha e_{\mu}f +(1-\alpha) e_{\mu}g.
\end{equation}
Combining \eqref{e:1} and \eqref{e:2} gives
$$\alpha e_{\mu}f +(1-\alpha) e_{\mu}g \leq\vphi\leq \alpha h_{\mu}f+(1-\alpha)h_{\mu}g
\leq \alpha f+(1-\alpha) g,$$
which is \eqref{e:three:c}. Equation \eqref{e:lim:0} follows from \eqref{e:three:c} by sending $\mu\downarrow 0$.
The pointwise and epigraphical limits agree because of
\cite[Proposition 7.4(d)]{rockwets}.

Now assume that $f,g$ are continuous.
Since both $e_{\mu}f$ and $f$ are continuous, and $e_{\mu}f\uparrow f$.
Dini's theorem says that $e_{\mu}f\uparrow f$ uniformly on compact subsets
of $\R^n$. The same can be said about
 $e_{\mu}g\uparrow g$.
Hence, the convergence in \eqref{e:lim:0} is uniform
on compact subsets of $\R^n$ by \eqref{e:three:c}.
\end{proof}

To study the epi-continuity of proximal average, we recall the following two standard notions.

\begin{df} A sequence of functions $(f_{k})_{k\in\N}$ is eventually prox-bounded
if there exists $\lambda>0$ such that $\liminf_{k\rightarrow\infty}e_{\lambda}f_{k}(x)>-\infty$
for some $x$. The supremum of all such $\lambda$ is then
the threshold of eventual prox-boundedness of the sequence.
\end{df}

\begin{df}
A sequence of functions $(f_{k})_{k\in \N}$ converges continuously to $f$ if
$f_{k}(x_{k})\rightarrow f(x)$ whenever $x_{k}\rightarrow x$.
\end{df}

The following key result is implicit in the proof of \cite[Theorem 7.37]{rockwets}.
We provide its proof for completeness.
Define $\aN=\{N\subset\N|\ \N\setminus N \text{ is finite}\}.$
\begin{lem}\label{l:env:cont}
 Let $(f_{k})_{k\in\N}$ and $f$ be proper, lsc functions on $\R^n$.
Suppose that $(f_{k})_{k\in\N}$ is eventually prox-bounded,
$\bar{\lambda}$ is the threshhold of eventual prox-boundedness,
and $f_{k}\epi f$. Suppose also that
$\mu_{k}, \mu\in ]0,\bar{\lambda}[$, and $\mu_{k}\rightarrow\mu$.
Then $f$ is prox-bounded with threshold $\lambda_{f}\geq \bar{\lambda}$, and
$e_{\mu_{k}}f_{k}$ converges continuously to $e_{\mu}f$.
In particular,
$e_{\mu_{k}}f_{k}\epi e_{\mu}f$, and $e_{\mu_{k}}f_{k}\pw e_{\mu}f$.
\end{lem}
\begin{proof} Let $\varepsilon \in ]0,\bl[$. The eventual prox-boundness of $(f_{k})_{k\in\N}$
means that
there exist $b\in\R^n$,
$\beta\in\R$ and $N\in\aN$ such that
$$(\forall k\in N)(\forall w\in \R^n)\ f_{k}(w)
\geq \beta-\frac{1}{2\varepsilon}\|b-w\|^2.$$
Let $\mu\in ]0,\varepsilon[$.
Consider any $x\in\R^n$ and any sequence $x_{k}\rightarrow x$
in $\R^n$, any sequence $\mu_{k}\rightarrow\mu$ in $(0,\bl)$.
Since $f_{k}\epi f$, the functions
$f_{k}+(1/2\mu_{k})\|\cdot-x_{k}\|^2$ epi-converge to $f+(1/2\mu)\|\cdot-x\|^2$.
Take $\delta\in ]\mu,\varepsilon[$. Because
$\mu_{k}\rightarrow\mu$, there exists $N'\subseteq N$, $N'\in\aN$
such that $\mu_{k}\in (0,\delta)$ when $k\in N'$.
Then $\forall k\in N'$,
\begin{align*}
f_{k}(w)+\frac{1}{2\mu_{k}}\|x_{k}-w\|^2 &\geq \beta-\frac{1}{2\varepsilon}\|b-w\|^2+ \frac{1}{2\delta}\|x_{k}-w\|^2\\
&=\beta-\frac{1}{2\varepsilon}\|b-w\|^2+ \frac{1}{2\delta}\|(x_{k}-b)+(b-w)\|^2\\
&\geq \beta+\bigg(\frac{1}{2\delta}-\frac{1}{2\varepsilon}\bigg)
\|b-w\|^2-\frac{1}{\delta}\|x_{k}-b\|\|b-w\|.
\end{align*}
In view of $x_{k}\rightarrow x$, the sequence $(\|x_{k}-b\|)_{k\in\N}$ is bounded,
say by $\rho>0$. We have
$$
(\forall k\in N')\ f_{k}(w)+\frac{1}{2\mu_{k}}\|x_{k}-w\|^2 \geq h(w):=
\beta+\bigg(\frac{1}{2\delta}-\frac{1}{2\varepsilon}\bigg)\|b-w\|^2-\frac{\rho}{\delta}
\|b-w\|.
$$
The function $h$ is level-bounded because $\delta<\varepsilon$. Hence, by
\cite[Theorem 7.33]{rockwets},
$$\lim_{k\rightarrow\infty}\inf_{w}\bigg(f_{k}(w)+\frac{1}{2\mu_{k}}
\|x_{k}-w\|^2\bigg)
=\inf_{w}\bigg(f(w)+\frac{1}{2\mu}\|x-w\|^2\bigg),$$
i.e., $e_{\mu_{k}}f_{k}(x_{k})\rightarrow e_{\mu}f(x)$.
Also, $e_{\mu}f(x)$ is finite, so $\lambda_{f}\geq \mu$.
Since $\varepsilon\in ]0,\bl[$ and $\mu\in ]0,\varepsilon[$ were
arbitrary, the result holds whenever $\mu\in ]0,\bl[$. This in turn implies
$\lambda_{f}\geq\bl$.
\end{proof}

For the convenience of analyzing the full epi-continuity,
below we write the proximal average $\vphi$ explicitly in the form
$\vphif$.

\begin{thm}[full epi-continuity of proximal average]
Let the sequences of functions $(f_{k})_{k\in \N}$, $(g_{k})_{k\in\N}$ on $\R^n$
 be eventually prox-bounded
with threshold of eventual prox-boundedness $\bar{\lambda}>0$.
Let $(\mu_{k})_{k\in\N}$ be a sequence and $\mu$ in
$]0,\bar{\lambda}[$ and let $(\alpha_{k})_{k\in\N}$ be a sequence and $\alpha$
in
$[0,1]$. Suppose that $f_{k}\epi f$, $g_{k}\epi g$,
$\mu_{k}\rightarrow\mu$, and $\alpha_{k}\rightarrow\alpha$.
Then $\vphifk\epi\vphif$.
\end{thm}
\begin{proof} Consider any $x\in\R^n$ and any sequence $x_{k}\rightarrow x$.
By \cite[Example 11.26]{rockwets},
$$e_{\mu_{k}}f_{k}(\mu_{k}x_{k})=\frac{\mu_{k}\|x_{k}\|^2}{2}-
\bigg(f_{k}+\frac{1}{2\mu_{k}}\|\cdot\|^2\bigg)^*(x_{k}).$$
Lemma~\ref{l:env:cont} shows that
\begin{align*}
\lim_{k\rightarrow\infty}
\bigg(f_{k}+\frac{1}{2\mu_{k}}\|\cdot\|^2\bigg)^*(x_{k})
& = \lim_{k\rightarrow\infty} \frac{\mu_{k}\|x_{k}\|^2}{2}-e_{\mu_{k}}f_{k}(\mu_{k}x_{k})
=\frac{\mu\|x\|^2}{2}-e_{\mu}f(\mu x)\\
&=\bigg(f+\frac{1}{2\mu}\|\cdot\|^2\bigg)^*(x).
\end{align*}
Therefore, the functions $\left(f_{k}+\frac{1}{2\mu_{k}}\|\cdot\|^2\right)^*$ converge
continuously to $\left(f+\frac{1}{2\mu}\|\cdot\|^2\right)^*$.
It follows that
$$\alpha_{k}\bigg(f_{k}+\frac{1}{2\mu_{k}}\|\cdot\|^2\bigg)^*+
(1-\alpha_{k})\bigg(g_{k}+\frac{1}{2\mu_{k}}\|\cdot\|^2\bigg)^*$$
converges continuously
to
$$\alpha\bigg(f+\frac{1}{2\mu}\|\cdot\|^2\bigg)^*+
(1-\alpha)\bigg(g+\frac{1}{2\mu}\|\cdot\|^2\bigg)^*,$$
so epi-converges.
Then by Wijsman's theorem \cite[Theorem 11.34]{rockwets},
$$\bigg[\alpha_{k}\bigg(f_{k}+\frac{1}{2\mu_{k}}\|\cdot\|^2\bigg)^*+
(1-\alpha_{k})\bigg(g_{k}+\frac{1}{2\mu_{k}}\|\cdot\|^2\bigg)^*\bigg]^*$$
epi-converges to
$$\bigg[\alpha\bigg(f+\frac{1}{2\mu}\|\cdot\|^2\bigg)^*+
(1-\alpha)\bigg(g+\frac{1}{2\mu}\|\cdot\|^2\bigg)^*\bigg]^*.$$
Since $(\mu, x)\mapsto \frac{1}{2\mu}\|x\|^2$ is continuous
on $]0,+\infty[\times\R^n$, we have that
$$\vphifk
=\bigg[\alpha_{k}\bigg(f_{k}+\frac{1}{2\mu_{k}}\|\cdot\|^2\bigg)^*+
(1-\alpha_{k})\bigg(g_{k}+\frac{1}{2\mu_{k}}\|\cdot\|^2\bigg)^*\bigg]^*
-\frac{1}{2\mu_{k}}\|\cdot\|^2$$
epi-converges to
$$\vphif
=\bigg[\alpha\bigg(f+\frac{1}{2\mu}\|\cdot\|^2\bigg)^*+
(1-\alpha)\bigg(g+\frac{1}{2\mu}\|\cdot\|^2\bigg)^*\bigg]^*
-\frac{1}{2\mu}\|\cdot\|^2.$$
\end{proof}


\begin{cor}[epi-continuity of the proximal average]\label{t:alpha:epi}
Let $0<\mu<\bl$.
Then the function
$\alpha\mapsto \vphi$ is continuous with respect to the epi-topology. That is,
$\forall (\alpha_{k})_{k\in\N}$ and $\alpha$ in $[0,1]$,
$$\alpha_{k}\rightarrow\alpha \quad \Rightarrow \quad \vphik\epi\vphi.$$
In particular, $\vphi\epi h_{\mu}g$ when $\alpha\downarrow 0$, and
$\vphi\epi h_{\mu}f$ when $\alpha\uparrow 1$.
\end{cor}

\section{Optimal value and minimizers of the proximal average}\label{s:opti}

\subsection{Relationship of infimum and minimizers among
$\vphi$, $f$ and $g$.}

\begin{prop}\label{p:minimizer} Let $0<\mu<\bl$. One has
\begin{enumerate}[label=\rm(\alph*)]
\item \label{e:env:conv}
\begin{align*}
\inf\vphi &= \inf[\alpha e_{\mu}f+(1-\alpha)e_{\mu}g], \text{ \emph{and} }\\
\argmin\vphi &=\argmin [\alpha e_{\mu}f+(1-\alpha)e_{\mu}g];
\end{align*}
\item \label{e:hull:arith}
\begin{align*}
\alpha\inf f+(1-\alpha)\inf g& \leq \inf\vphi
\leq\inf[\alpha h_{\mu}f+(1-\alpha)h_{\mu}g]
\leq \inf[\alpha f+(1-\alpha) g].
\end{align*}
\end{enumerate}
\end{prop}

\begin{proof}
For \ref{e:env:conv}, apply Theorem~\ref{t:prox}\ref{i:env:conhull}
and $\argmin\vphi=\argmin e_{\mu}\vphi$.
For \ref{e:hull:arith}, apply Theorem~\ref{t:u:0}\ref{i:three:c}
and $\inf e_{\mu}f=\inf f$, and $\inf e_{\mu}g=\inf g$.
\end{proof}

\begin{thm} Suppose that $\argmin f\cap \argmin g\neq\varnothing$ and
$\alpha\in ]0,1[$.
Then the following hold:
\begin{enumerate}[label=\rm(\alph*)]
\item \label{i:arithmin}
\begin{equation}\label{e:arith}
\min(\alpha f+(1-\alpha)g)=\alpha\min f+(1-\alpha)\min g, \text{ \emph{and} }
\end{equation}
\begin{equation}\label{e:arith:minimizer}
\argmin(\alpha f+(1-\alpha)g)=\argmin f\cap \argmin g;
\end{equation}
\item \label{i:proxmin}
\begin{equation}\label{e:minvalue}
\min\vphi=\alpha\min f+(1-\alpha)\min g, \text{ \emph{and} }
\end{equation}
\begin{equation*}
\argmin\vphi=\argmin f\cap \argmin g.
\end{equation*}
\end{enumerate}
\end{thm}
\begin{proof} Pick $x\in \argmin f\cap \argmin g$. We have
\begin{equation}\label{e:minvalue:two}
\inf[\alpha f+(1-\alpha)g]=\alpha f(x)+(1-\alpha) g(x)=\alpha\min f+(1-\alpha)\min g.
\end{equation}
\ref{i:arithmin}: Equation \eqref{e:minvalue:two} gives \eqref{e:arith} and
\begin{equation*}\label{e:setinclus}
(\argmin f\cap\argmin g)\subseteq \argmin (\alpha f+(1-\alpha)g).
\end{equation*}
 To see the converse inclusion of \eqref{e:minvalue},
let $x\in\argmin(\alpha f+(1-\alpha)g)$. Then \eqref{e:arith} gives
$$\alpha \min f+(1-\alpha)\min g=\min(\alpha f+(1-\alpha)g)=\alpha f(x)+(1-\alpha)g(x),$$
from which
$$\alpha (\min f-f(x))+(1-\alpha)(\min g-g(x))=0.$$
Since $\min f\leq f(x), \min g\leq g(x)$, we obtain
 $\min f= f(x), \min g= g(x)$, so $x\in\argmin f\cap\argmin g$. Thus,
 $\argmin (\alpha f+(1-\alpha)g)\subseteq (\argmin f\cap\argmin g)$. Hence,
 \eqref{e:arith:minimizer} holds.

\noindent\ref{i:proxmin}: Equation \eqref{e:minvalue} follows from Proposition~\ref{p:minimizer} and Theorem~\ref{t:u:0}\ref{i:three:c}. This also gives
$$(\argmin f\cap \argmin g)\subseteq \argmin\vphi.$$

To show $(\argmin f\cap \argmin g)\supseteq \argmin\vphi$, take any $x\in \argmin\vphi$.
By \eqref{e:minvalue} and Theorem~\ref{t:u:0}\ref{i:three:c}, we have
$$\alpha\min f+(1-\alpha)\min g=\vphi(x)\geq \alpha e_{\mu}f(x)+(1-\alpha)e_{\mu}g(x),$$
from which
$$\alpha(e_{\mu}f(x)-\min f)+(1-\alpha)(e_{\mu}g(x)-\min g)\leq 0.$$
Since $\min f=\min e_{\mu}f$ and $\min g=\min e_{\mu}g$, it follows that
$e_{\mu}f(x)=\min e_{\mu}f$ and $e_{\mu}g(x)=\min e_{\mu}g$, so
$x\in(\argmin e_{\mu}f\cap\argmin e_{\mu}g)=(\argmin f\cap\argmin g)$
because of $\argmin e_{\mu}f=\argmin f$ and $\argmin e_{\mu}g=\argmin g$.
\end{proof}

To explore further optimization properties of $\vphi$, we need the following three
auxiliary results.
\begin{lem}\label{l:box} Suppose that $f_{1}, f_{2}:\R^n\rightarrow\RX$ are
proper and lsc, and that
$f_{1}\Box f_{2}$ is exact. Then
\begin{enumerate}[label=\rm(\alph*)]
\item\label{i:box:inf}
\begin{equation}\label{e:epis:inf}
\inf (\fot)=\inf f_{1}+\inf f_{2}, \text{ and }
\end{equation}
\item\label{i:box:min}
\begin{equation}\label{e:epis:min}
\argmin (\fot)=\argmin f_{1}+\argmin f_{2}.
\end{equation}
\end{enumerate}
\end{lem}
\begin{proof} Equation \eqref{e:epis:inf} follows from
\begin{align*}
\inf\fot &=\inf_{x}\inf_{x=y+z}[f_{1}(y)+f_{2}(z)]\\
&=\inf_{y,z}[f_{1}(y)+f_{2}(z)]=\inf f_{1}+\inf f_{2}.
\end{align*}
To see \eqref{e:epis:min}, we first show
\begin{equation}\label{e:setsum}
\argmin (\fot)\subseteq\argmin f_{1}+\argmin f_{2}.
\end{equation}
If $\argmin (\fot)=\varnothing$, the inclusion holds trivially. Let us
assume that $\argmin (\fot)\neq \varnothing$ and
let $x\in\argmin (\fot)$. Since $\fot$ is exact, we have $x=y+z$ for some $y,z$  and
$\fot(x)=f_{1}(y)+f_{2}(z)$. In view of \eqref{e:epis:inf},
\begin{align*}
f_{1}(y)+f_{2}(z) &=\fot(x)=\min \fot=\inf f_{1}+\inf f_{2},
\end{align*}
from which
$$(f_{1}(y)-\inf f_{1})+(f_{2}(z)-\inf f_{2})=0.$$
Then $f_{1}(y)=\inf f_{1}, f_{2}(z)=\inf f_{2}$, which gives
$y\in\argmin f_{1}, z\in\argmin f_{2}$. Therefore,
$x\in \argmin f_{1}+\argmin f_{2}$. Next, we show
\begin{equation}\label{e:setsum2}
\argmin (\fot)\supseteq\argmin f_{1}+\argmin f_{2}.
\end{equation}
If one of $\argmin f_{1}, \argmin f_{2}$ is empty, the inclusion holds
trivially. Assume that
$\argmin f_{1}\neq\varnothing$ and $\argmin f_{2}\neq\varnothing$.
Take
$y\in\argmin f_{1}, z\in\argmin f_{2}$, and put $x=y+z$. The definition of
$\Box$ and \eqref{e:epis:inf} give
$$\fot(x)\leq f_{1}(y)+f_{2}(z)=\min f_{1}+\min f_{2}=
\inf \fot,$$
which implies $x\in\argmin(\fot)$. Since $y\in\argmin f_{1}$, $z\in\argmin f_{2}$
were arbitrary,
\eqref{e:setsum2} follows. Combining \eqref{e:setsum} and \eqref{e:setsum2}
gives \eqref{e:epis:min}.
\end{proof}

\begin{lem}\label{l:epimulti}
 Let $f_{1}:\R^n\rightarrow\RX$ be proper and lsc, and let $\beta>0$. Then
 \begin{enumerate}[label=\rm(\alph*)]
 \item\label{i:epimul:inf}
$$\inf\left[\beta f_{1}\left(\frac{\cdot}{\beta}\right)\right]=\beta\inf f_{1},\text{ and }
$$
\item \label{i:epimul:min}
$$\argmin \left[\beta f_{1}
\left(\frac{\cdot}{\beta}\right)\right]=\beta\argmin f_{1}.$$
\end{enumerate}
\end{lem}

\begin{lem}\label{l:convexhull:f} Let $f_{1}:\R^n\rightarrow\RX$ be proper and lsc.
Then the following
hold:
\begin{enumerate}[label=\rm(\alph*)]
\item\label{i:convexh:inf}
$$\inf (\conv f_1) =\inf f_{1};$$
\item\label{i:convexh:min}
if, in addition, $f_{1}$ is coercive, then
$$\argmin  (\conv f_{1})
=\conv(\argmin f_{1}),$$
and $\argmin  (\conv f_{1})\neq\varnothing$.
\end{enumerate}
\end{lem}
\begin{proof} Combine \cite[Comment 3.7(4)]{benoist} and
\cite[Corollary 3.47]{rockwets}.
\end{proof}

We are now ready for the main result of this section.

\begin{thm}\label{t:shifted}
Let $0<\mu<\bl$, and
let $\vphi$ be defined as in \eqref{e:prox:def}.
Then the following hold:
\begin{enumerate}[label=\rm(\alph*)]
\item\label{i:shifted:inf}
\begin{align*}
&\inf\left(\vphi+\frac{1}{2\mu}\|\cdot\|^2\right)\\
&=\alpha\inf\left(
f+\frac{1}{2\mu}\|\cdot\|^2\right)+(1-\alpha)
\inf \left(
g+\frac{1}{2\mu}\|\cdot\|^2\right);
\end{align*}
\item\label{i:shifted:min}
\begin{align*}
&\argmin\left(\vphi+\frac{1}{2\mu}\|\cdot\|^2\right)\\
& =\alpha\conv \left[\argmin\left(
f+\frac{1}{2\mu}\|\cdot\|^2\right)\right]+(1-\alpha)
\conv \left[\argmin\left(
g+\frac{1}{2\mu}\|\cdot\|^2\right)\right]\neq\varnothing.
\end{align*}
\end{enumerate}
\end{thm}
\begin{proof} Theorem~\ref{t:prox}\ref{i:epi:sum} gives
\begin{align*}
& \vphi+\frac{1}{2\mu}\|\cdot\|^2\\
&=\left[\alpha\conv\bigg(f+\frac{1}{2\mu}\|\cdot\|^2\bigg)\left(\frac{\cdot}{\alpha}\right)\right]
\Box \left[(1-\alpha)\conv
\bigg(g+\frac{1}{2\mu}\|\cdot\|^2\bigg)\left(\frac{\cdot}{1-\alpha}\right)\right],
\end{align*}
in which the inf-convolution $\Box$ is exact.

\noindent\ref{i:shifted:inf}: Using Lemma \ref{l:box}\ref{i:box:inf} and Lemma \ref{l:convexhull:f}\ref{i:convexh:inf},
we deduce
\begin{align*}
&\inf\left(\vphi+\frac{1}{2\mu}\|\cdot\|^2\right)\\
&=\inf \left[\alpha\conv\bigg(f+\frac{1}{2\mu}\|\cdot\|^2\bigg)\left(\frac{\cdot}{\alpha}\right)\right]
+\inf \left[(1-\alpha)\conv
\bigg(g+\frac{1}{2\mu}\|\cdot\|^2\bigg)\left(\frac{\cdot}{1-\alpha}\right)\right]\\
&=\alpha\inf \left[\conv\bigg(f+\frac{1}{2\mu}\|\cdot\|^2\bigg)\right]
+(1-\alpha)\inf \left[\conv
\bigg(g+\frac{1}{2\mu}\|\cdot\|^2\bigg)\right]\\
&=\alpha\inf\bigg(f+\frac{1}{2\mu}\|\cdot\|^2\bigg)
+(1-\alpha)\inf
\bigg(g+\frac{1}{2\mu}\|\cdot\|^2\bigg).
\end{align*}

\noindent\ref{i:shifted:min}: Note that
$f+\frac{1}{2\mu}\|\cdot\|^2$ and
$g+\frac{1}{2\mu}\|\cdot\|^2$ are coercive
because of $0<\mu<\bl$.
Using Lemma~\ref{l:box}\ref{i:box:min}-Lemma~\ref{l:convexhull:f}\ref{i:convexh:min},
we deduce
\begin{align*}
&\argmin\left(\vphi+\frac{1}{2\mu}\|\cdot\|^2\right)\\
&=\argmin \left[\alpha\conv\bigg(f+\frac{1}{2\mu}\|\cdot\|^2\bigg)\left(\frac{\cdot}{\alpha}\right)\right]
+\argmin \left[(1-\alpha)\conv
\bigg(g+\frac{1}{2\mu}\|\cdot\|^2\bigg)\left(\frac{\cdot}{1-\alpha}\right)\right]\\
&=\alpha\argmin \left[\conv\bigg(f+\frac{1}{2\mu}\|\cdot\|^2\bigg)\right]
+(1-\alpha)\argmin \left[\conv
\bigg(g+\frac{1}{2\mu}\|\cdot\|^2\bigg)\right]\\
&=\alpha\conv\left[\argmin \bigg(f+\frac{1}{2\mu}\|\cdot\|^2\bigg)\right]
+(1-\alpha)\conv\left[\argmin
\bigg(g+\frac{1}{2\mu}\|\cdot\|^2\bigg)\right].
\end{align*}
Finally, these three sets of minimizers are nonempty by
Lemma~\ref{l:convexhull:f}\ref{i:convexh:min}.
\end{proof}

\begin{rem} \emph{Theorem~\ref{t:shifted}\ref{i:shifted:min}} is just a rewritten form of
$$\Prox_{\mu}\vphi(0)=\alpha\conv[\Prox_{\mu}f(0)]+(1-\alpha)\conv[\Prox_{\mu}g(0)].$$
\end{rem}

In view of
Theorem~\ref{t:go:infinity}\ref{i:inf:phi2},
when $\bl=\infty$, as $\mu\rightarrow\infty$ the pointwise limit is
$$\vphi\pw \left[
\alpha\conv f\bigg(\frac{\cdot}{\alpha}\bigg)\Box (1-\alpha)\conv
g\bigg(\frac{\cdot}{1-\alpha}\bigg)\right],$$
and the epi-limit is
$$\vphi\epi \cl\left[
\alpha\conv f\bigg(\frac{\cdot}{\alpha}\bigg)\Box (1-\alpha)\conv
g\bigg(\frac{\cdot}{1-\alpha}\bigg)\right].$$

We conclude this section with a result on minimization of this limit.
\begin{prop}\label{p:coercive}
Suppose that both $f$ and $g$ are coercive. Then the following hold:
\begin{enumerate}[label=\rm(\alph*)]
\item\label{i:coercive:1}
 $\alpha\conv f\left(\frac{\cdot}{\alpha}\right)\Box (1-\alpha)\conv
g\left(\frac{\cdot}{1-\alpha}\right)$ is proper, lsc and convex;
\item\label{i:coercive:2}
$$\min \left[\alpha\conv f\left(\frac{\cdot}{\alpha}\right)\Box (1-\alpha)\conv
g\left(\frac{\cdot}{1-\alpha}\right)\right]=
\alpha\min f+(1-\alpha)\min g;$$
\item\label{i:coercive:3}
\begin{align*}
& \argmin\left[\alpha\conv f\left(\frac{\cdot}{\alpha}\right)\Box (1-\alpha)\conv
g\left(\frac{\cdot}{1-\alpha}\right)\right]\\
& =
\alpha\conv\argmin f+(1-\alpha)\conv\argmin g\neq\varnothing.
\end{align*}
\end{enumerate}
\end{prop}
\begin{proof} Since both $f$ and $g$ are coercive, by \cite[Corollary 3.47]{rockwets},
$\conv f$ and $\conv g$ are lsc, convex and coercive.
As
$$(\alpha f^*+(1-\alpha) g^*)^*=\cl\left[\alpha\conv f\left(\frac{\cdot}{\alpha}\right)\Box (1-\alpha)\conv
g\left(\frac{\cdot}{1-\alpha}\right)\right]$$
and $\dom f^*=\R^n=\dom g^*$, the closure operation on the right-hand side is superfluous.
This establishes \ref{i:coercive:1}.
Moreover, the
infimal convolution
\begin{equation}\label{e:coercive0}
\alpha\conv f\left(\frac{\cdot}{\alpha}\right)\Box (1-\alpha)\conv
g\left(\frac{\cdot}{1-\alpha}\right)
\end{equation}
is exact.
For
\ref{i:coercive:2}, \ref{i:coercive:3}, it suffices to apply Lemma~\ref{l:box} to
\eqref{e:coercive0} for functions $\alpha\conv f\left(\frac{\cdot}{\alpha}\right)$ and
$\alpha\conv g\left(\frac{\cdot}{\alpha}\right)$, followed by invoking
Lemma~\ref{l:epimulti} and Lemma~\ref{l:convexhull:f}.
\end{proof}

\subsection{Convergence in minimization}

We need the following result on coercivity.

\begin{lem}\label{l:psi}
 Let $0<\mu<\bl$, and let $\psi:\R^n\rightarrow\R$ be a convex function. If
$f\geq \psi, g\geq\psi$, then $\vphi\geq \psi$.
\end{lem}
\begin{proof}
Recall
$\vphi(x)=$
$$\left[\alpha\conv\bigg(f+\frac{1}{2\mu}\|\cdot\|^2\bigg)\left(\frac{\cdot}{\alpha}\right)\Box (1-\alpha)\conv
\bigg(g+\frac{1}{2\mu}\|\cdot\|^2\bigg)\left(\frac{\cdot}{1-\alpha}\right)\right](x)
-\frac{1}{2\mu}\|x\|^2.
$$
As $f+\frac{1}{2\mu}\|\cdot\|^2\geq \psi+\frac{1}{2\mu}\|\cdot\|^2$ and the latter is convex,
we have
$$\conv\left(f+\frac{1}{2\mu}\|\cdot\|^2\right)\geq \psi+\frac{1}{2\mu}\|\cdot\|^2;$$
similarly,
$$\conv\left(g+\frac{1}{2\mu}\|\cdot\|^2\right)\geq \psi+\frac{1}{2\mu}\|\cdot\|^2.$$
Then
\begin{align*}
&\alpha\conv\bigg(f+\frac{1}{2\mu}\|\cdot\|^2\bigg)\left(\frac{\cdot}{\alpha}\right)\Box (1-\alpha)\conv
\bigg(g+\frac{1}{2\mu}\|\cdot\|^2\bigg)\left(\frac{\cdot}{1-\alpha}\right)\\
&\geq \alpha\bigg(\psi+\frac{1}{2\mu}\|\cdot\|^2\bigg)\left(\frac{\cdot}{\alpha}\right)
\Box
(1-\alpha)\bigg(\psi+\frac{1}{2\mu}\|\cdot\|^2\bigg)\left(\frac{\cdot}{1-\alpha}\right)\\
&=\psi+\frac{1}{2\mu}\|\cdot\|^2,
\end{align*}
in which we have used the convexity of $\psi+\frac{1}{2\mu}\|\cdot\|^2$.
The result follows.
\end{proof}

\begin{thm} Let $0<\mu<\bl$. One has the following.
\begin{enumerate}[label=\rm(\alph*)]
\item \label{i:coercive0}If $f, g$ are bounded from below,
then $\vphi$ is bounded from below.
\item\label{i:coercive1}
 If $f, g$ are level-coercive, then $\vphi$ is level-coercive.
\item\label{i:coercive2}
 If $f,g$ are coercive, then $\vphi$ is coercive.
\end{enumerate}
\end{thm}
\begin{proof}
\ref{i:coercive0}: Put $\psi=\min\{\inf f, \inf g\}$ and apply Lemma~\ref{l:psi}.

\noindent\ref{i:coercive1}:
By \cite[Theorem 3.26(a)]{rockwets}, there exist $\gamma\in (0,\infty)$, and $\beta\in\R$
such that $f\geq \psi, g\geq \psi$ with $\psi=\gamma\|\cdot\|+\beta$.
Apply Lemma~\ref{l:psi}.

\noindent\ref{i:coercive2}: By \cite[Theorem 3.26(b)]{rockwets}, for
every $\gamma\in (0,\infty)$, there exists $\beta\in\R$
such that $f\geq \psi, g\geq \psi$ with $\psi=\gamma\|\cdot\|+\beta$.
Apply Lemma~\ref{l:psi}.
\end{proof}
\begin{thm} Suppose that the proper, lsc functions $f, g$ are level-coercive. Then
for every $\balpha\in [0,1]$, we have
\begin{align*}
\lim_{\alpha\rightarrow\balpha}\inf\vphi & =\inf \vphiba \text{ (finite)}, \text{ and }
\\
\limsup_{\alpha\rightarrow\balpha}\argmin\vphi & \subseteq\argmin \vphiba.
\end{align*}
Moreover, $(\argmin\vphi)_{\alpha\in[0,1]}$ lies in a bounded set.
Consequently,
$$\lim_{\alpha\downarrow 0}\inf\vphi=\inf g, \text{ and } \limsup_{\alpha\downarrow 0}\argmin\vphi\subseteq\argmin g;$$
$$\lim_{\alpha\uparrow 1}\inf\vphi=\inf f, \text{ and }
\limsup_{\alpha\uparrow 1}\argmin\vphi\subseteq\argmin f.$$
\end{thm}
\begin{proof} By assumption, there exist $\gamma>0$ and $\beta\in\R$ such that
$f\geq\gamma\|\cdot\|+\beta, g\geq\gamma\|\cdot\|+\beta$. Lemma~\ref{l:psi}
shows that $\vphi\geq \gamma\|\cdot\|+\beta$ for every $\alpha\in [0,1]$.
Since $\gamma\|\cdot\|+\beta$ is level-bounded,
 $(\vphi)_{\alpha\in [0,1]}$ is uniformly level-bounded (so
eventually level-bounded). Corollary~\ref{t:alpha:epi} says that
$\alpha\mapsto \vphi$ is epi-continuous on $[0,1]$.
As $\lambda_{f}=\lambda_{g}=\infty$,
$\vphi$ and $\vphiba$ are proper and lsc for every $\mu>0$. Hence
\cite[Theorem 7.33]{rockwets} applies.
\end{proof}
\begin{thm} Suppose that the proper, lsc functions $f, g$ are level-coercive and
$\dom f\cap\dom g\neq\varnothing$. Then
\begin{align*}
\lim_{\mu\downarrow 0}\inf\vphi & =\inf (\alpha f+(1-\alpha)g), \text{ and }
\\
\limsup_{\mu\downarrow 0}\argmin\vphi & \subseteq\argmin (\alpha f+(1-\alpha)g).
\end{align*}
Moreover, $(\argmin\vphi)_{\mu>0}$ lies in a bounded set.
\end{thm}
\begin{proof}
Note that each $\vphi$ is proper and lsc, and $f+g$ is proper and lsc.
By Theorem~\ref{t:u:0}, when $\mu\downarrow 0$, $\vphi$ epi-converges to
$f+g$.
By assumption, there exist $\gamma>0$ and $\beta\in\R$ such that
$f\geq\gamma\|\cdot\|+\beta, g\geq\gamma\|\cdot\|+\beta$. Lemma~\ref{l:psi}
shows that $\vphi\geq \gamma\|\cdot\|+\beta$ for every $\mu\in ]0,\infty[$.
Since $\gamma\|\cdot\|+\beta$ is level-bounded,
 $(\vphi)_{\mu\in ]0,\infty[}$ is uniformly level-bounded (so
eventually level-bounded).
It remains to
apply \cite[Theorem 7.33]{rockwets}.
\end{proof}

\begin{thm} Suppose that the proper and lsc functions $f, g$ are coercive.
Then
for every $\bmu\in ]0,\infty]$, we have
\begin{align}\label{e:bcoercive}
\lim_{\mu\uparrow\bmu}\inf\vphi & =\inf \vphibmu \text{ (finite)}, \text{ and }
\nonumber\\
\limsup_{\mu\uparrow \bmu}\argmin\vphi & \subseteq\argmin \vphibmu.
\end{align}
Moreover, $(\argmin\vphi)_{\mu>0}$ lies in a bounded set.
Consequently,
\begin{align}\label{e:coercive:inf}
\lim_{\mu\uparrow \infty}\inf\vphi & =\alpha\min f+(1-\alpha)\min g, \text{ and }
\nonumber \\
\limsup_{\mu\uparrow \infty}\argmin\vphi & \subseteq
(\alpha\conv\argmin f+(1-\alpha)\conv\argmin g).
\end{align}
\end{thm}
\begin{proof} Note that each $\vphi$ is proper and lsc for $\mu\in]0,\infty[$.
When $\mu=\infty$, Proposition~\ref{p:coercive} gives that
the epi-limit is proper, lsc and convex. By Theorem~\ref{t:go:infinity}\ref{i:mono:phi},
when $\mu\uparrow\bmu$, $\vphi$ monotonically decrease to $\vphibmu$.
Since $\vphibmu$ is lsc, so $\vphi$ epi-converges to $\vphibmu$.
By assumption, for every $\gamma>0$ there exists $\beta\in\R$ such that
$f\geq\gamma\|\cdot\|+\beta, g\geq\gamma\|\cdot\|+\beta$. Lemma~\ref{l:psi}
shows that $\vphi\geq \gamma\|\cdot\|+\beta$ for every $\mu\in ]0,\infty[$.
Since $\gamma\|\cdot\|+\beta$ is level-bounded,
 $(\vphi)_{\mu\in ]0,\infty[}$ is uniformly level-bounded (so
eventually level-bounded). Hence \eqref{e:bcoercive} follows from
\cite[Theorem 7.33]{rockwets}.
Combining \eqref{e:bcoercive}, Theorem~\ref{t:go:infinity}
and Proposition~\ref{p:coercive} yields \eqref{e:coercive:inf}.
\end{proof}

\section{Subdifferentiability of the proximal average}\label{s:subd}
In this section, we focus on the subdifferentiability and differentiability of proximal average.

Following Benoist and Hiriart-Urruty \cite{benoist},
we say that a family of points $\{x_{1},\ldots, x_{m}\}$ in $\dom f$
is
called by $x\in\dom\conv f$ if
$$x=\sum_{i=1}^{m}\alpha_{i}x_{i}, \text{ and } \conv f(x)=\sum_{i=1}^{m}\alpha_{i}f(x_{i}),$$
where $\sum_{i=1}^{m}\alpha_{i}=1$ and $(\forall i)\ \alpha_{i}>0$.
The following result is the central one of this section.
\begin{thm}[subdifferentiability of the proximal average]\label{t:vphi:sub}
Let
$0<\mu<\bl$,
let $x\in\dom\vphi$ and $x=y+z$.
Suppose the following conditions hold:
\begin{enumerate}[label=\rm(\alph*)]
\item \label{i:function1}
\begin{align*}
&\left[\alpha\conv\bigg(f+\frac{1}{2\mu}\|\cdot\|^2\bigg)\left(\frac{\cdot}{\alpha}\right)\Box (1-\alpha)\conv
\bigg(g+\frac{1}{2\mu}\|\cdot\|^2\bigg)\left(\frac{\cdot}{1-\alpha}\right)\right](x)\\
&=\alpha\conv\bigg(f+\frac{1}{2\mu}\|\cdot\|^2\bigg)\left(\frac{y}{\alpha}\right)+(1-\alpha)\conv
\bigg(g+\frac{1}{2\mu}\|\cdot\|^2\bigg)\left(\frac{z}{1-\alpha}\right),
\end{align*}
\item  \label{i:function2}$\{y_{1},\ldots,y_{l}\}$ are called by
$y/\alpha$ in $\conv(f+1/2\mu\|\cdot\|^2)$, and
\item  \label{i:function3}
$\{z_{1},\ldots,z_{m}\}$ are called by
$z/(1-\alpha)$ in $\conv(g+1/2\mu\|\cdot\|^2)$.
\end{enumerate}
Then
\begin{align*}
\hat{\partial}\vphi(x) &=
\partial_{L}\vphi(x)
=\partial_{C}\vphi(x)
\\
&=\left[\cap_{i=1}^{l}\partial\left(f+\frac{1}{2\mu}\|\cdot\|^2\right)
\left(y_{i}\right)\right]\cap
\left[\cap_{j=1}^{m}\partial\left(g+\frac{1}{2\mu}\|\cdot\|^2\right)\left(z_{j}\right)\right]-
\frac{x}{\mu}.
\end{align*}
\end{thm}

\begin{proof} By Theorem~\ref{t:prox}\ref{i:epi:sum}, the Clarke regularity of
$\vphi$ and sum rule of limiting subdifferentials, we have
$\hat{\partial}\vphi(x)=\partial_{C}\vphi(x)=\partial_{L}\vphi(x)=$
\begin{align}\label{e:diff:conv}
&\partial_{L}\left[\alpha\conv\bigg(f+\frac{1}{2\mu}\|\cdot\|^2\bigg)\left(\frac{\cdot}{\alpha}\right)\Box (1-\alpha)\conv
\bigg(g+\frac{1}{2\mu}\|\cdot\|^2\bigg)\left(\frac{\cdot}{1-\alpha}\right)\right](x)
-\frac{x}{\mu}\\
&=\partial \left[\alpha\conv\bigg(f+\frac{1}{2\mu}\|\cdot\|^2\bigg)\left(\frac{\cdot}{\alpha}\right)\Box (1-\alpha)\conv
\bigg(g+\frac{1}{2\mu}\|\cdot\|^2\bigg)\left(\frac{\cdot}{1-\alpha}\right)\right](x)
-\frac{x}{\mu}.\nonumber
\end{align}
Using the subdifferential formula for infimal convolution \cite[Proposition 16.61]{convmono} or \cite[Corollary 2.4.7]{zalinescu2002convex}, we obtain
\begin{align*}\label{e:infsub}
& \partial\left[\alpha\conv\bigg(f+\frac{1}{2\mu}\|\cdot\|^2\bigg)\left(\frac{\cdot}{\alpha}\right)\Box (1-\alpha)\conv
\bigg(g+\frac{1}{2\mu}\|\cdot\|^2\bigg)\left(\frac{\cdot}{1-\alpha}\right)\right](x)\\
&=
\partial \left[\alpha\conv\bigg(f+\frac{1}{2\mu}\|\cdot\|^2\bigg)\left(\frac{y}{\alpha}\right)\right]\cap
\partial \left[(1-\alpha)\conv
\bigg(g+\frac{1}{2\mu}\|\cdot\|^2\bigg)\left(\frac{z}{1-\alpha}\right)\right]\\
&=
\partial\conv\bigg(f+\frac{1}{2\mu}\|\cdot\|^2\bigg)(\bar{y})\cap
\partial\conv
\bigg(g+\frac{1}{2\mu}\|\cdot\|^2\bigg)(\bar{z})
\end{align*}
where $\bar{y}=\frac{y}{\alpha}$, $\bar{z}=\frac{z}{1-\alpha}$.
The subdifferential formula for the convex hull of a coercive function
\cite[Corollary 4.9]{benoist} or \cite[Theorem 3.2]{rabier} gives
\begin{equation*}\label{e:convf}
\partial\conv\bigg(f+\frac{1}{2\mu}\|\cdot\|^2\bigg)\big(\bar{y}\big)=
\cap_{i=1}^{l}\partial\bigg(f+\frac{1}{2\mu}\|\cdot\|^2\bigg)\big(y_{i}),
\end{equation*}
\begin{equation}\label{e:convg}
\partial\conv\bigg(g+\frac{1}{2\mu}\|\cdot\|^2\bigg)\big(\bar{z}\big)=
\cap_{j=1}^{m}\partial\bigg(g+\frac{1}{2\mu}\|\cdot\|^2\bigg)\big(z_{j}).
\end{equation}
Therefore, the result follows by combining \eqref{e:diff:conv} and \eqref{e:convg}.
\end{proof}

\begin{cor}
Let $0<\mu<\bl$,
let $\alpha_{i}>0, \beta_{j}>0$ with
$\sum_{i=1}^{l}\alpha_{i}=1, \sum_{j=1}^{m}\beta=1$ and let
$\alpha\in ]0,1[$.
Suppose that
\begin{equation*}\label{e:x=yz}
x=\alpha\sum_{i=1}^{l}\alpha_{i}y_{i}+(1-\alpha)\sum_{j=1}^{m}\beta_{j}z_{j},
\end{equation*}
and
\begin{equation}\label{e:commonsub}
\left[\cap_{i=1}^{l}\partial\left(f+\frac{1}{2\mu}\|\cdot\|^2\right)
\left(y_{i}\right)\right]\cap
\left[\cap_{j=1}^{m}\partial\left(g+\frac{1}{2\mu}\|\cdot\|^2\right)\left(z_{j}\right)\right]
\neq \varnothing.
\end{equation}
Then
\begin{align*}
\hat{\partial}\vphi(x) &=
\partial_{L}\vphi(x)
=\partial_{C}\vphi(x)\nonumber\\
&=\left[\cap_{i=1}^{l}\partial\left(f+\frac{1}{2\mu}\|\cdot\|^2\right)
\left(y_{i}\right)\right]\cap
\left[\cap_{j=1}^{m}\partial\left(g+\frac{1}{2\mu}\|\cdot\|^2\right)\left(z_{j}\right)\right]-
\frac{x}{\mu}.
\end{align*}
\end{cor}
\begin{proof} We will show that
\begin{equation}\label{e:called:y0}
\conv\left(f+\frac{1}{2\mu}\|\cdot\|^2\right)\sum_{i=1}^{l}\alpha_{i}y_{i}
=\sum_{i}^{l}\alpha_{i}
\left(f+\frac{1}{2\mu}\|\cdot\|^2\right)(y_{i}).
\end{equation}
By \eqref{e:commonsub}, there exists
$$y^*\in \left[\cap_{i=1}^{l}\partial\left(f+\frac{1}{2\mu}\|\cdot\|^2\right)
\left(y_{i}\right)\right]\cap
\left[\cap_{j=1}^{m}\partial\left(g+\frac{1}{2\mu}\|\cdot\|^2\right)\left(z_{j}\right)\right].$$
For every $y_{i}$, we have
$$(\forall u\in\R^n)\ \left(f+\frac{1}{2\mu}\|\cdot\|^2\right)(u)\geq \left(f+\frac{1}{2\mu}\|\cdot\|^2\right)(y_{i})
+\scal{y^*}{u-y_{i}}.$$ Multiplying each inequality by $\alpha_i$, followed by summing
them up, gives
$$(\forall u\in\R^n)\ \left(f+\frac{1}{2\mu}\|\cdot\|^2\right)(u)\geq \sum_{i=1}^{l}\alpha_{i}\left(f+\frac{1}{2\mu}\|\cdot\|^2\right)(y_{i})
+\scal{y^*}{u-\sum_{i=1}^{l}\alpha_{i}y_{i}}.$$
Then
$$(\forall u\in\R^n)\ \conv\left(f+\frac{1}{2\mu}\|\cdot\|^2\right)(u)\geq \sum_{i=1}^{l}\alpha_{i}\left(f+\frac{1}{2\mu}\|\cdot\|^2\right)(y_{i})
+\scal{y^*}{u-\sum_{i=1}^{l}\alpha_{i}y_{i}},$$
from which
\begin{equation}\label{e:called:y1}
\conv\left(f+\frac{1}{2\mu}\|\cdot\|^2\right)\sum_{i=1}^{l}\alpha_{i}y_{i}
\geq \sum_{i=1}^{l}\alpha_{i}\left(f+\frac{1}{2\mu}\|\cdot\|^2\right)(y_{i}).
\end{equation}
Since $\conv\left(f+\frac{1}{2\mu}\|\cdot\|^2\right)(\sum_{i=1}^{l}\alpha_{i}y_{i})
\leq \sum_{i=1}^{l}\alpha_{i}\left(f+\frac{1}{2\mu}\|\cdot\|^2\right)(y_{i})$
always holds, \eqref{e:called:y0} is established.
Moreover, \eqref{e:called:y0}
and \eqref{e:called:y1} implies
\begin{equation}\label{e:called:y2}
y^*\in \partial \conv\left(f+\frac{1}{2\mu}\|\cdot\|^2\right)
\sum_{i=1}^{l}\alpha_{i}y_{i}.
\end{equation}
Similar arguments give
\begin{equation}\label{e:called:z0}
\conv\left(g+\frac{1}{2\mu}\|\cdot\|^2\right)(\sum_{j=1}^{m}\beta_{j}z_{j})
=\sum_{j}^{m}\beta_{j}
\left(g+\frac{1}{2\mu}\|\cdot\|^2\right)(z_{j}),
\end{equation}
and
\begin{equation}\label{e:called:z1}
y^*\in \partial \conv\left(g+\frac{1}{2\mu}\|\cdot\|^2\right)\sum_{j=1}^{m}\beta_{j}y_{j}.
\end{equation}
Put
$x=y+z$ with $y=\alpha\sum_{i=1}^{l}\alpha_{i}y_{i}$ and $z=(1-\alpha)
\sum_{j=1}^{m}\beta_{j}z_{j}.$
Equations \eqref{e:called:y2} and \eqref{e:called:z1} guarantee
the assumption \ref{i:function1} of Theorem~\ref{t:vphi:sub}; \eqref{e:called:y0} and \eqref{e:called:z0} guarantee the assumptions \ref{i:function2} and \ref{i:function3}
of Theorem~\ref{t:vphi:sub} respectively.
Hence, Theorem~\ref{t:vphi:sub} applies.
\end{proof}

\begin{cor}
Let $0<\mu<\bl$.
Suppose that
$$\partial\left(f+\frac{1}{2\mu}\|\cdot\|^2\right)(x)\cap \partial\left(g+\frac{1}{2\mu}\|\cdot\|^2\right)(x)\neq\varnothing.$$
Then
\begin{align*}
\hat{\partial}\vphi(x) &=
\partial_{L}\vphi(x)
=\partial_{C}\vphi(x)\\
&=\partial\left(f+\frac{1}{2\mu}\|\cdot\|^2\right)
(x)\cap
\partial\left(g+\frac{1}{2\mu}\|\cdot\|^2\right)(x)-
\frac{x}{\mu}.
\end{align*}
\end{cor}

Armed with Theorem~\ref{t:vphi:sub}, we now turn to the differentiability of $\vphi$.
\begin{df}
A function $f_{1}:\R^n\rightarrow\RX$ is almost differentiable
if $\hat{\partial} f_{1}(x)$ is a singleton
for every $x\in \intt(\dom f_{1})$, and $\hat{\partial} f_{1}(x)=\varnothing$
for every $x\in\dom f_{1}\setminus\intt(\dom f_{1})$, if any.
\end{df}

\begin{lem}\label{l:sumrule}
Let $f_{1},f_{2}:\R^n\rightarrow\RX$ be proper, lsc functions
and let $x\in\dom f_{1}\cap\dom f_{2}$. If $f_{2}$ is
 continuously differentiable at $x$, then
$$\partial (f_{1}+f_{2})(x)
\subset\hat{\partial}(f_{1}+f_{2})(x)=
\hat{\partial}f_{1}(x)+\triangledown f_{2}(x).$$
\end{lem}
\begin{lem}\label{l:diff:hypo}
 Let $f_{1}:\R^n\rightarrow \RX$ be proper, lsc and $\mu$-proximal, and
let $x\in\intt\dom f_{1}$. If $\hat{\partial}f_{1}(x)$ is a singleton, then
$f_{1}$ is differentiable at $x$.
\end{lem}
\begin{proof} Observe that
$f_{2}=f_{1}+\frac{1}{2\mu}\|\cdot\|^2$ is convex, and
$$\partial f_{2}(x)=\hat{\partial}f_{2}(x)
=\hat{\partial}f_{1}(x)+\frac{x}{\mu}.$$
When $\hat{\partial}f_{1}(x)$ is a singleton,  $\partial f_{2}(x)$
is a singleton. This implies that $f_{2}$ is differentiable at $x$
because $f_{2}$ is convex and $x\in\intt\dom f_{2}$.
Hence, $f_{1}$ is differentiable at $x$.
\end{proof}

\begin{cor}[differentiability of the proximal average] \label{c:interior:diff}
Let $0<\mu<\bl$.
Suppose that either $f$ or $g$ is almost differentiable (in particular, if $f$ or $g$ is differentiable
at every point of its domain). Then
$\vphi$ is almost differentiable. In particular, $\vphi$ is differentiable on the interior of
its domain
$\intt\dom\vphi$.
\end{cor}
\begin{proof} Without loss of generality, assume that
$f$ is almost differentiable.
By Lemma~\ref{l:sumrule},
\begin{equation}\label{e:frechet}
\partial\left(f+\frac{1}{2\mu}\|\cdot\|^2\right)
\left(y_{i}\right)\subset \hat{\partial} f(y_{i})+\frac{y_{i}}{\mu}.
\end{equation}
It follows that
$\partial\left(f+\frac{1}{2\mu}\|\cdot\|^2\right)(y_{i})$ is
at most single-valued whenever
$\hat{\partial}f(y_{i})$ is single-valued.

With the same notation as in Theorem~\ref{t:vphi:sub}, we consider two cases.

{\sl Case 1:} $x\in\bdry\dom\vphi$. As $x=\alpha(y/\alpha)+(1-\alpha)(z/(1-\alpha))$,
we must have
$y/\alpha\in(\bdry\conv \dom f)$ and $z/(1-\alpha)\in\bdry(\conv \dom g)$;
otherwise $x\in\intt(\alpha\conv\dom f+(1-\alpha)\conv\dom g)
=\intt\dom\vphi$,
which is a contradiction. Then the family of $\{y_{1},\ldots, y_{m}\}$ called
by $y/\alpha$ must be from $\bdry\dom f$. As $f$ is almost differentiable,
$\hat{\partial}f(y_{i})=\varnothing$, then $\hat{\partial}\vphi(x)=\varnothing$ by
Theorem~\ref{t:vphi:sub} and \eqref{e:frechet}.

{\sl Case 2:} $x\in\intt(\dom\vphi)$. As $\vphi$ is $\mu$-proximal,
$\hat{\partial}\vphi(x)\neq\varnothing$. We claim that
the family of $\{y_{1},\ldots, y_{m}\}$ called by
$y/\alpha$ in Theorem~\ref{t:vphi:sub} are necessarily from $\intt\dom f$.
If not, then
$\partial\left(f+\frac{1}{2\mu}\|\cdot\|^2\right)(y_{i})=\varnothing$
because of \eqref{e:frechet} and
$\hat{\partial}f(y_{i})=\varnothing$ for $y_{i}\in\bdry(\dom f)$.
Then Theorem~\ref{t:vphi:sub} implies
$\hat{\partial}\vphi(x)=\varnothing$, which is a contradiction.
Now $\{y_{1},\ldots, y_{m}\}$ are from $\intt\dom f$ and $f$ is almost
differentiable, so $(\forall i)\ \hat{\partial}f(y_{i})$
is a singleton. Using \eqref{e:frechet} again and
$\hat{\partial}\vphi(x)\neq\varnothing$,
we see that $(\forall i)\ \partial\left(f+\frac{1}{2\mu}\|\cdot\|^2\right)(y_{i})$
is a singleton. Hence, $\hat{\partial}\vphi(x)$ is a singleton by
Theorem~\ref{t:vphi:sub}.

Case 1 and Case 2 together show that $\vphi$ is almost differentiable.
Finally, $\vphi$ is differentiable on
$\intt\dom\vphi$ by Lemma~\ref{l:diff:hypo}.
\end{proof}

\begin{cor} Let $0<\mu<\bl$.
Suppose that either $f$ or $g$ is almost differentiable and that either $\conv\dom f=\R^n$
or $\conv\dom g=\R^n$. Then
$\vphi$ is differentiable on $\R^n$.
\end{cor}
\begin{proof}
By Theorem~\ref{t:prox}\ref{i:dom:convhull}, $\dom\vphi=\R^n$.
It suffices to apply Corollary~\ref{c:interior:diff}.
\end{proof}

We end this section with a result on Lipschitz continuity of the gradient of $\vphi$.
\begin{prop} Suppose that $f$ (or $g$) is differentiable with
a Lipschtiz continuous gradient and $\mu$-proximal.
Then, for every $\alpha\in ]0,1[$, the function $\vphi$ is differentiable
with a Lipschitz continuous gradient.
\end{prop}
\begin{proof}
As $f$ is $\mu$-proximal and differentiable with a Lipschtiz continuous
gradient,
the function $f+\frac{1}{2\mu}\|\cdot\|^2$ is convex and differentiable with a Lipschitz
continuous gradient. By \cite[Proposition 12.60]{rockwets},
$\big(f+\frac{1}{2\mu}\|\cdot\|^2\big)^*$ is strongly convex,
so $$\alpha\bigg(f+\frac{1}{2\mu}\|\cdot\|^2\bigg)^*+
(1-\alpha)\bigg(g+\frac{1}{2\mu}\|\cdot\|^2\bigg)^*$$ is strongly convex.
By \cite[Proposition 12.60]{rockwets} again,
$$\left[\alpha\bigg(f+\frac{1}{2\mu}\|\cdot\|^2\bigg)^*+
(1-\alpha)\bigg(g+\frac{1}{2\mu}\|\cdot\|^2\bigg)^*\right]^*$$
is convex and differentiable with a Lipschitz continuous gradient.
Since
$\vphi=$
$$\left[\alpha\bigg(f+\frac{1}{2\mu}\|\cdot\|^2\bigg)^*+
(1-\alpha)\bigg(g+\frac{1}{2\mu}\|\cdot\|^2\bigg)^*\right]^*
-\frac{1}{2\mu}\|\cdot\|^2,$$
we see that $\vphi$ is differentiable with a Lipschitz continuous gradient.
\end{proof}

\section{The proximal average for quadratic functions}\label{s:quad}
In this section, we illustrate the above results for quadratic functions.
For an $n\times n $ symmetric matrix $A$, define the quadratic function
$\jj_{A}:\R^n\rightarrow\R$ by $x\mapsto \frac{1}{2}\scal{x}{Ax}.$
We use $\lmin A$ to denote the smallest eigenvalue of $A$.
\begin{lem}\label{l:quad} For an $n\times n$ symmetric matrix $A$, one has
\begin{enumerate}[label=\rm(\alph*)]
\item\label{i:quad1}
 $\jj_{A}$ is prox-bounded with threshold
\begin{equation}\label{e:boundp}
\lambda_{\jj_{A}}=\frac{1}{\max\{0,-\lmin A\}}>0
\end{equation}
and $\mu$-proximal for every $0<\mu\leq \lambda_{\jj_{A}}$;
\item \label{e:boundp1}
the prox-bound $\lambda_{\jj_{A}}=+\infty$ if and only if $A$ is positive semidefinite;
\item\label{i:quad2} if $0<\mu<\lambda_{\jj_{A}}$, then
\begin{equation}\label{e:quadenv}
e_{\mu}\jj_{A}=\jj_{\mu^{-1}[\Id-(\mu A+\Id)^{-1}]};\text { and }
\end{equation}
\begin{equation*}\label{e:quadprox}
\Prox_{\mu}\jj_{A}=(\mu A+\Id)^{-1}.
\end{equation*}
\end{enumerate}
\end{lem}
\begin{proof}
\ref{i:quad1}: As $A$ can be diagonalized,
$\jj_{A}\geq \lmin A \jj_{\Id}$. Apply \cite[Exercise 1.24]{rockwets} to
obtain \eqref{e:boundp}. When $0<\mu\leq \lambda_{\jj_{A}}$,
$A+\frac{1}{\mu}\Id$ has nonnegative  eigenvalues,
so $\jj_{A}+\frac{1}{\mu}\jj_{\Id}$ is convex.

\noindent\ref{e:boundp1}: This follows from \ref{i:quad1}.

\noindent\ref{i:quad2}: When $0<\mu<\lambda_{\jj_{A}}$, the function
$\jj_{A}+\frac{1}{\mu}\jj_{\Id}$ is strictly convex. To find
\begin{equation}\label{e:quade}
e_{\mu}\jj_{A}(x)=\inf_{w}\left(\jj_{A}(w)+\frac{1}{\mu}\jj_{\Id}(x-w)\right),
\end{equation}
one directly takes derivative with repect to $w$ to find
\begin{equation}\label{e:quadp}
\Prox_{\mu}\jj_{A}(x)=(\mu A+\Id)^{-1}(x).
\end{equation}
Substitute \eqref{e:quadp} into \eqref{e:quade} to get \eqref{e:quadenv}.
\end{proof}

\begin{ex} Let $A_{1}, A_{2}$ be two $n\times n$ symmetric matrices and
let $0<\mu<\bl=\min\{\lambda_{\jj_{A_{1}}},\lambda_{\jj_{A_{2}}}\}$.
Then the following hold:
\begin{enumerate}[label=\rm(\alph*)]
\item\label{e:proxq1}
$\vphi=\jj_{\mu^{-1}A_{3}}$ with
$$A_{3}=[\alpha(\mu A_{1}+\Id)^{-1}+(1-\alpha)(\mu A_{2}+\Id)^{-1}]^{-1}-\Id,$$
and
$$\Prox_{\mu}\vphi=\alpha(\mu A_{1}+\Id)^{-1}+(1-\alpha)(\mu A_{2}+\Id)^{-1};$$
\item \label{e:proxq2}
$\lim_{\alpha\downarrow 0}\vphi=\jj_{A_{2}}$ and $\lim_{\alpha\uparrow 1}\vphi=\jj_{A_{1}};$

\item\label{e:proxq3}
$\lim_{\mu\downarrow 0}\vphi=\alpha\jj_{A_{1}}+(1-\alpha)\jj_{A_{2}};$

\item \label{e:proxq3.5} when $\bl<\infty$,
$$\lim_{\mu\uparrow \bl}\vphi=
\jj_{\alpha^{-1}(A_{1}+\bl^{-1}\Id)}\Box \jj_{(1-\alpha)^{-1}(A_{2}+\bl^{-1}\Id)}-
\jj_{\bl^{-1}\Id};$$
\item \label{e:proxq4}
when both $A_{1}, A_{2}$ are positive definite, $\bl=+\infty$,
$$\lim_{\mu\uparrow \infty}\vphi=\jj_{(\alpha A_{1}^{-1}+(1-\alpha)A_{2}^{-1})^{-1}}.$$
\end{enumerate}
\end{ex}
\begin{proof}
\ref{e:proxq1}: By Lemma~\ref{l:quad},
\begin{align*}
& -\alpha e_{\mu}\jj_{A_{1}}-(1-\alpha)e_{\mu}\jj_{A_{2}}\\
&=-\alpha \jj_{\mu^{-1}[\Id-(\mu A_{1}+\Id)^{-1}]}-(1-\alpha)\jj_{\mu^{-1}[\Id-(\mu A_{2}+\Id)^{-1}]}\\
&=\jj_{\mu^{-1}[(\alpha(\mu A_{1}+\Id)^{-1}+(1-\alpha)(\mu A_{2}+\Id)^{-1})-\Id]}.
\end{align*}
Thus, applying Lemma~\ref{l:quad} again,
\begin{align*}
\vphi &=-e_{\mu}(\jj_{\mu^{-1}[\alpha(\mu A_{1}+\Id)^{-1}+(1-\alpha)(\mu A_{2}+\Id)^{-1}-\Id]})
\\
&=-\jj_{\mu^{-1}[\Id-(\alpha(\mu A_{1}+\Id)^{-1}+(1-\alpha)(\mu A_{2}+\Id)^{-1})^{-1}]}\\
&=\jj_{\mu^{-1}[(\alpha(\mu A_{1}+\Id)^{-1}+(1-\alpha)(\mu A_{2}+\Id)^{-1})^{-1}-\Id]}.
\end{align*}
Again, using  Lemma~\ref{l:quad},
\begin{align*}
e_{\mu}\vphi &= e_{\mu}\jj_{\mu^{-1}[(\alpha(\mu A_{1}+\Id)^{-1}+(1-\alpha)(\mu A_{2}+\Id)^{-1})^{-1}-\Id]}\\
&=\jj_{\mu^{-1}[\Id-(\alpha(\mu A_{1}+\Id)^{-1}+(1-\alpha)(\mu A_{2}+\Id)^{-1})]},
\end{align*}
so
$$\Prox_{\mu}\vphi=\alpha(\mu A_{1}+\Id)^{-1}+(1-\alpha)(\mu A_{2}+\Id)^{-1}.$$

\noindent\ref{e:proxq2}: Note that the matrix function $A\mapsto A^{-1}$ is continuous whenever
$A$ is invertible. Then \ref{e:proxq2} is immediate because
$$\lim_{\alpha\downarrow 0}(\alpha(\mu A_{1}+\Id)^{-1}+(1-\alpha)(\mu A_{2}+\Id)^{-1})^{-1}
=((\mu A_{2}+\Id)^{-1})^{-1}=\mu A_{2}+\Id, \text{ and }
$$
$$\lim_{\alpha\uparrow 1}(\alpha(\mu A_{1}+\Id)^{-1}+(1-\alpha)(\mu A_{2}+\Id)^{-1})^{-1}
=((\mu A_{1}+\Id)^{-1})^{-1}=\mu A_{1}+\Id.
$$

\noindent\ref{e:proxq3}: It suffices to show
$$\lim_{\mu\downarrow 0}\frac{[\alpha(\mu A_{1}+\Id)^{-1}+(1-\alpha)(\mu A_{2}+\Id)^{-1}]^{-1}-\Id}{\mu}=\alpha A_{1}+(1-\alpha)A_{2},$$
equivalently,
\begin{equation}\label{e:matrix}
\small\lim_{\mu\downarrow 0}\frac{[\alpha(\mu A_{1}+\Id)^{-1}+(1-\alpha)(\mu A_{2}+\Id)^{-1}]^{-1}-[\alpha(\mu A_{1}+\Id)+(1-\alpha)(\mu A_{2}+\Id)]}{\mu}=0.
\end{equation}
Since $\lim_{\mu\downarrow 0}[\alpha(\mu A_{1}+\Id)^{-1}+(1-\alpha)(\mu A_{2}+\Id)^{-1}]
=\Id$, \eqref{e:matrix} follows from the following calculation:
\begin{align*}
&[\alpha(\mu A_{1}+\Id)^{-1}+(1-\alpha)(\mu A_{2}+\Id)^{-1}]\cdot
\nonumber\\
&\frac{[\alpha(\mu A_{1}+\Id)^{-1}+(1-\alpha)(\mu A_{2}+\Id)^{-1}]^{-1}-[\alpha(\mu A_{1}+\Id)+(1-\alpha)(\mu A_{2}+\Id)]}{\mu}\\
&= \frac{\Id-[\alpha(\mu A_{1}+\Id)^{-1}+(1-\alpha)(\mu A_{2}+\Id)^{-1}]
[\alpha(\mu A_{1}+\Id)+(1-\alpha)(\mu A_{2}+\Id)]}{\mu}\\
&=-\alpha(1-\alpha)[(\mu A_{1}+\Id)^{-1}(A_{2}-A_{1})
+(\mu A_{2}+\Id)^{-1}(A_{1}-A_{2})]\\
& \rightarrow -\alpha(1-\alpha)[(A_{2}-A_{1})
+(A_{1}-A_{2})]= 0.
\end{align*}

\noindent\ref{e:proxq3.5}: The matrices $A_{1}+\bl^{-1}\Id$ and $A_{2}+\bl^{-1}\Id$
are positive semidefinite, so the convex hulls are superfluous.

\noindent\ref{e:proxq4}: As $\mu\rightarrow\infty$, we have
\begin{align*}
&\frac{[\alpha(\mu A_{1}+\Id)^{-1}+(1-\alpha)(\mu A_{2}+\Id)^{-1}]^{-1}-\Id}{\mu}\\
&=\left[\alpha\left(A_{1}+\frac{\Id}{\mu}\right)^{-1}
+(1-\alpha)\left(A_{2}+\frac{\Id}{\mu}\right)^{-1}\right]-\frac{\Id}{\mu}\\
&\rightarrow (\alpha A_{1}^{-1}+(1-\alpha)A_{2}^{-1})^{-1}.
\end{align*}
\end{proof}

\begin{rem} When both $A_{1}, A_{2}$ are positive semidefinite matrices,
we refer the reader to \emph{\cite{respos}}.
\end{rem}

\section{The general question is still unanswered}\label{s:theg}
According to Theorem~\ref{prop:convcomb}, suppose that $0<\mu<\bl$,
$0<\alpha<1$
and $\Prox_{\mu}f$ and $\Prox_{\mu}g$ are convex-valued. Then
there exists a proper, lsc function
$\vphi$ such that
$\Prox_{\mu}\vphi=\alpha\Prox_{\mu}f+(1-\alpha)\Prox_{\mu}g$.
When the proximal mapping is not convex-valued, the situation is subtle.
We illustrate this by revisiting Example~\ref{e:proximal:fk}.
Recall that for $\varepsilon_{k}>0$, the function
$$f_{k}(x)=\max\{0,(1+\varepsilon_{k})(1-x^2)\}$$
has
$$\Prox_{1/2}f_{k}(x)=\begin{cases}
x &\text{ if $x\geq 1$,}\\
1 &\text{ if $0<x<1$,}\\
\{-1,1\} &\text{ if $x=0$,}\\
-1 &\text{ if $-1<x<0$,}\\
x &\text{ if $x\leq -1$.}
\end{cases}
$$
With $\alpha=1/2$, we have
\begin{equation}\label{e:prox:half}
(\alpha \Prox_{1/2}f_{1}+(1-\alpha)\Prox_{1/2}f_{2})(x)
=\begin{cases}
x &\text{ if $x\geq 1$,}\\
1 &\text{ if $0<x<1$,}\\
\{-1,0,1\} &\text{ if $x=0$,}\\
-1 &\text{ if $-1<x<0$,}\\
x &\text{ if $x\leq -1$.}
\end{cases}
\end{equation}
Because $\Prox_{1/2}f_{i}(0)$ is not convex-valued, $(\alpha \Prox_{1/2}f_{1}+(1-\alpha)\Prox_{1/2}f_{2})(0)$ is neither
$\Prox_{1/2}f_{1}(0)$ nor $\Prox_{1/2}f_{2}(0)$, although
$\Prox_{1/2}f_{1}(0)=\Prox_{1/2}f_{2}(0)$.

One can verify that \eqref{e:prox:half} is indeed
$\Prox_{1/2}g(x)$ where
$$g(x)=\begin{cases}
0 &\text{ if $x>1$,}\\
-x(x-1)-x^2+1 &\text{ if $0<x\leq 1$,}\\
-x(x+1)-x^2+1 &\text{ if $-1<x\leq 0$,}\\
0 &\text{ if $x\leq -1$.}
\end{cases}
$$
Regretfully, we do not have a systematic way to find $g$
when $\Prox_{\mu}g$ is not convex-valued.
The challenging question is still open:

\emph{Is a convex combination of proximal mappings of possibly nonconvex functions
always a proximal mapping?}

\section*{Acknowledgment}
Xianfu Wang was partially supported by the Natural Sciences and
Engineering Research Council of Canada.


\bibliographystyle{plain}
\bibliography{Bibliography}{}
\end{document}